\documentclass[a4paper, reqno]{article} %a4paper

\usepackage[utf8]{inputenc}
\usepackage[T1]{fontenc}

\usepackage[english]{babel}

\usepackage[margin=1.2in]{geometry}
\usepackage{enumitem}
\usepackage[bookmarks = true]{hyperref}
\usepackage{xcolor}
\usepackage{amsmath, amsthm, amssymb, bbm}
\usepackage[abbrev]{amsrefs}
\usepackage{unicode-math}
\usepackage[small,nohug,heads=vee]{diagrams}
\diagramstyle[labelstyle=\scriptstyle]

\theoremstyle{plain}
\newtheorem{thm}{Theorem}[section]
\newtheorem*{thm*}{Theorem}
\newtheorem{lem}[thm]{Lemma}
\newtheorem{prop}[thm]{Proposition}
\newtheorem{cor}[thm]{Corollary}

\newtheorem{propdef}[thm]{Proposition/Definition}

\theoremstyle{definition}
\newtheorem{rmk}[thm]{Remark}

\newtheorem{defn}[thm]{Definition}
\newtheorem{ex}[thm]{Example}

\newtheorem{construction}[thm]{Construction}

\numberwithin{equation}{section}

\newcommand{\overbar}[1]{\mkern 1.5mu\overline{\mkern-1.5mu#1\mkern-1.5mu}\mkern 1.5mu}

\newcommand{\ob}{\overbar}

\newcommand{\mr}{\mathrm}

\newcommand{\wt}{\widetilde}

\newcommand{\Supp}{\operatorname{Supp}}

\newcommand{\codim}{\operatorname{codim}}

\newcommand{\tensor}{\otimes}

\newcommand{\iso}{\cong}

\newcommand{\mbC}{\mathbb{C}}

\newcommand{\mbN}{\mathbb{N}}

\newcommand{\mbQ}{\mathbb{Q}}
\newcommand{\mbR}{\mathbb{R}}

\newcommand{\mbZ}{\mathbb{Z}}

\newcommand{\mcS}{\mathcal{S}}
\newcommand{\mcT}{\mathcal{T}}

\begin{document}

\title{On Tropical Intersection Theory}
\author{Andreas Mihatsch}
\date{August 28, 2022}
\maketitle
\setlength{\parskip}{-1mm}
\tableofcontents
\setlength{\parskip}{0mm}

\begin{abstract}
We develop a tropical intersection formalism of forms and currents that extends classical tropical intersection theory in two ways. First, it allows to work with arbitrary polyhedra, also non-rational ones. Second, it allows for smooth differential forms as coefficients. The intersection product in our formalism can be defined through the diagonal intersection method of Allermann--Rau or the fan displacement rule. We prove with a limiting argument that both definitions agree.
\end{abstract}

\section{Introduction}
\label{ss:intro}
For their ``Tropical approach to non-archimedean Arakelov theory'' \cite{GK}, Gubler--Künnemann combine tropical intersection theory and smooth differential forms into their formalism of so-called \emph{$δ$-forms}.
They use these to develop a calculus of Green currents on non-archimedean spaces that is related to intersection theory on formal models.
The strength of their approach is that $δ$-forms are simpler to work with than formal models, leading to a computationally accessible handle for certain arithmetic intersection problems.

The present paper contributes to these ideas through the development of a more general and concise theory of $δ$-forms.
This is a purely tropical endeavor: $δ$-Forms are a natural generalization of tropical cycles and have the same formal properties.
For example, they admit pull-backs, push-forwards and a tropical intersection product called the $\wedge$-product.
$δ$-Forms also encompass Lagerberg's smooth forms \cite{Lag} and obey the same kind of differential calculus.
They furthermore come with a \emph{boundary operator} that generalizes the frequently used corner locus constructions of Esterov \cite{Esterov} and Francois \cite{Francois}, also cf. Gubler--Künnemann \cite{GK}.
Moreover, our formalism allows \emph{non-rational} polyhedra throughout. For tropical cycles, this generalization had already been obtained by Esterov \cite{Esterov}.

We now provide a more detailed description of $δ$-forms and our results. \emph{Smooth forms} are always meant in the sense of Lagerberg in the following, cf. \cite{Lag} or §\ref{ss:smooth_forms}. Recall that a \emph{current} is a continuous linear form on the space of smooth forms with compact support. A smooth form $α$ on $\mbR^n$ and a polyhedron $σ\subseteq \mbR^n$ define a \emph{current of integration} $(α\wedge σ)(η) := \int_{σ} α\wedge η$. (The polyhedron really needs to be weighted for this to work which will be explained below.) A current is called \emph{polyhedral} if it is a locally finite sum $\sum_{i\in I}α_i \wedge σ_i$ of such integration currents. In particular, polyhedral currents are entirely combinatorial objects. The following is our main definition.
\begin{defn}\label{def:intro_currents}
A \emph{$δ$-form} on $\mbR^n$ is a polyhedral current $T$ on $\mbR^n$ such that both derivatives $d'T$ and $d''T$ are again polyhedral.
\end{defn}
The differentials $d'$ and $d''$ here are taken in the sense of currents, i.e. as the duals of $d'$ and $d''$ for smooth forms. $δ$-Forms turn out to be stable under $d'$ and $d''$. Additional structure is then provided by defining a $δ$-form $T = \sum_{i\in I} α_i\wedge σ_i$ to be \emph{of tridegree $(p,q,r)$} if the $α_i$ may be chosen of bidegree $(p,q)$ and the $σ_i$ of codimension $r$. Then $d'$ naturally decomposes as $d' = d'_P - \partial'$, where $d'_P$ is trihomogeneous of tridegree $(1, 0, 0)$ and $\partial'$ trihomogeneous of tridegree $(0, -1, 1)$. The first summand $d'_P$ is the so-called \emph{polyhedral derivative} $d'_P(α\wedge σ) = (d'α)\wedge σ$ of Gubler--Künnemann, while $\partial'$ is the above-mentioned boundary operator. The latter is closely related to boundary integration of differential forms and to the corner locus construction, cf. \eqref{eq:div_inter_simplified}. A similar decomposition $d'' = d''_P - \partial''$ exists for $d''$.

Next, we come to the combinatorial description of $δ$-forms.
\begin{thm}\label{thm:intro}
A polyhedral current $T = \sum_{i\in I} α_i\wedge σ_i$ is a $δ$-form if and only if the datum $(α_i, σ_i)_{i\in I}$ is \emph{balanced} in the sense of tropical geometry.
\end{thm}
We formulate the relevant balancing condition in \eqref{eq:intro_balanced} below. Note that Thm. \ref{thm:intro} has precursors in the literature: Lagerberg \cite{Lag}*{Prop. 4.7}, Gubler \cite{Gub_forms_currents}*{Prop. 3.8} and Gubler--Künnemann \cite{GK}*{Prop. 2.16} (in successive level of generality) essentially prove it whenever the $α_i$ are smooth functions. Cast in our terminology, they show that the tropical cycles of codimension $r$ with smooth coefficients are exactly the $δ$-forms of tridegree $(0,0,r)$.

Thm. \ref{thm:intro} makes $δ$-forms behave like tropical cycles and we show that Allermann--Rau's construction of an intersection product \cite{AR} goes through without substantial change. This leads to our main result which is clearly inspired by Gubler--Künnemann's \cite{GK}*{Prop. 4.15}.
\begin{thm}\label{thm:intro2}
There is a graded-commutative $\wedge$-product of $δ$-forms that extends the $\wedge$-product of smooth forms and the intersection product of tropical cycles. The derivatives $d'$, $d''$, the polyhedral derivatives $d'_P$, $d''_P$ and the boundary derivatives $\partial'$, $\partial''$ all satisfy the Leibniz rule for $\wedge$.
\end{thm}
A more precise characterization of the $\wedge$-product may be found in the main text, cf. Thm. \ref{thm:main}. We also show that the $\wedge$-product can be computed by the \emph{fan displacement rule}, cf. Prop. \ref{prop:fan_displacement}. Recall that for intersections of tropical cycles, this rule goes back to Fulton--Sturmfels \cite{FS_tropical} and Mikhalkin \cite{Mikhalkin}. Its equality with Allermann--Rau's intersection product was shown independently by Rau \cite{Rau_thesis} and Katz \cite{Katz}. Our proof is similar to the combinatorial one of Rau and based on the observation that the $\wedge$-product suitably commutes with limits, cf. §\ref{ss:fan_displacement}.

We next explain the tropical formalism for possibly non-rational polyhedra. For a polyhedron $σ\subseteq \mbR^n$, denote by $N_σ\subseteq \mbR^n$ the linear space spanned by all $x-y,\ x,y\in σ$. Given a facet $τ\subset σ$, the subspace $N_τ\subset N_σ$ is of codimension $1$. If $σ$ is \emph{rational}, then $(N_τ\cap \mbZ^n) \subset (N_σ\cap \mbZ^n)$ is a sublattice of corank $1$ and a \emph{normal vector} for $τ\subset σ$ is any vector $n_{σ,τ}\in N_σ\cap \mbZ^n$ that generates $(N_σ\cap \mbZ^n) / (N_τ\cap \mbZ^n)$ and points in direction of $σ$.
For the general situation, we consider \emph{weighted} polyhedra instead. A weight for $σ$ is simply a generator $µ_σ\in \det N_σ$ \emph{up to sign}. Equivalently, it is a choice of Haar measure on $N_σ$. Given a facet inclusion $τ\subset σ$ and respective weights $µ_τ$ and $µ_σ$, a normal vector is any $n_{σ,τ}\in N_σ$ that satisfies $µ_σ = µ_τ\wedge n_{σ,τ}$ and points in direction of $σ$.
The two definitions are linked by the observation that every rational polyhedron $σ$ has a natural weight, namely the unique-up-to-sign generator of $\det_{\mbZ} (N_σ\cap \mbZ^n)$. The balancing condition \eqref{eq:intro_balanced} in Thm. \ref{thm:intro} is now a literal adaption of the classical balancing condition.
\begin{defn}\label{def:intro_balanced}
Consider a polyhedral complex $\mcT$, weights $(µ_σ)_{σ\in \mcT}$ for its polyhedra and smooth forms $(α_σ)_{σ\in \mcT}$, $α_σ\in A(σ)$. Here, $A(σ)$ denotes the smooth forms on $σ$. This datum is called \emph{balanced} if for all $τ\in \mcT$,
\begin{equation}\label{eq:intro_balanced}
\sum_{σ\in \mcT,\ τ\subset σ\ \text{a facet}} α_σ\vert_τ \tensor n_{σ,τ}\ \ \ \text{lies in}\ A(τ)\tensor_{\mbR} N_τ.
\end{equation}
\end{defn}
We next elucidate on the intersection theory of weighted polyhedra. Recall that given two properly intersecting \emph{rationally defined} subspaces $N_1,N_2\subseteq \mbR^n$, one defines their intersection multiplicity as the lattice index $\left[\mbZ^n:(N_1\cap \mbZ^n) + (N_2\cap \mbZ^n)\right]$. In the not necessarily rational case, still assuming proper intersection, one instead considers weights $µ_1$, $µ_2$ for $N_1$, $N_2$ and endows the intersection $N_1\cap N_2$ with the unique weight $ν$ such that $µ_1\tensor µ_2 = ν\tensor µ_{\mr{std}}$ under the canonical-up-to-sign identification $\det(V_1 \oplus V_2) = \det\left((V_1\cap V_2)\oplus \mbR^n\right)$. Here $µ_{\mr{std}}$ is the standard weight on $\mbR^n$. This rule extends to a full description of the $\wedge$-product of transversally intersecting $δ$-forms and underlies the fan displacement rule.

Finally, a weight $µ$ for $σ$ is also the precise datum needed to define the integral $\int_{[σ,µ]}η$ of a (compactly supported) form $η$ over $σ$. So in the definition of polyhedral current above, all polyhedra were silently weighted. For this natural reason, weights implicitly occur in Lagerberg \cite{Lag} and Chambert-Loir--Ducros \cite{CLD}. In fact, the \emph{calibrages} from \cite{CLD} are the same as our weights with an additional sign.

Tropical intersection theory has also been extended from $\mbR^n$ to more general combinatorial spaces. We will not address such questions here but take them up in our related work \cite{Mih_delta_nonarch}. More precisely, we develop there a theory of $δ$-forms on so-called tropical spaces with applications to non-archimedean Arakelov theory.

\subsubsection*{Layout}

§\ref{s:forms_and_currents} contains a summary of Lagerberg's theory of differential forms and introduces the formalism of weights, normal vectors and fiber integration. §\ref{s:delta_forms} is dedicated to the definition of $δ$-forms and to the proof of Thm. \ref{thm:intro}. §\ref{s:intersection_theory} contains the main result Thm. \ref{thm:intro2} and some additional properties of $δ$-forms. The fan displacement rule is Prop. \ref{prop:fan_displacement} and will be proved in §\ref{ss:fan_displacement}.

\subsection*{Acknowledgements}

I am grateful to W. Gubler, K. Künnemann and P. Scholze for comments on an earlier draft of the present article. I thank the referee for a careful reading of the article and many helpful suggestions.

\section{Forms and Currents}
\label{s:forms_and_currents}
\subsection{Smooth Forms}
\label{ss:smooth_forms}

Let $C^\infty(\mbR^n)$ and $\Omega^p(\mbR^n)$ denote the smooth functions and ``usual'' real smooth $p$-forms on $\mbR^n$. We fix the hosting space $\mbR^n$ for now and simply write $C^\infty$ and $\Omega^p$. The \emph{smooth forms} in this paper, whose definition is due to Lagerberg \cite{Lag}, are the elements of the exterior algebra
\begin{equation}\label{eq:def_pq_form}
A := A(\mbR^n) := \bigwedge\nolimits^{\!*}_{\, C^\infty} \big(\Omega^1\oplus \Omega^1\big).
\end{equation}
There is a bigrading $A = \bigoplus_{p,q} A^{p,q}$, where $A^{p,q}$ is the piece $\Omega^p\tensor_{C^\infty} \Omega^q$. Elements $α\in A^{p,q}$ are called bihomogeneous of bidegree $(p,q)$ and homogeneous of degree $\deg α = p + q$.

Being an exterior algebra, $A$ is endowed with a natural $\wedge$-product. It is bihomogeneous in the sense that $A^{p,q}\wedge A^{s,t}\subseteq A^{p+s, q+t}$. It is also graded-commutative, meaning that
\begin{equation}\label{eq:graded_commutativity_forms}
α\wedge β = (-1)^{\deg α \deg β} β\wedge α
\end{equation}
whenever $α$ and $β$ are homogeneous.

We use the terminology of \cite{CLD} for differential operators. Write $d_{\mr{std}}\colon  C^\infty \to \Omega^1$ for the usual differential. Given $f\in C^\infty$, we put
\begin{equation}\label{eq:def_d}
d'f = \left(d_{\mr{std}} f,\, 0\right),\ \ \ d''f = \left(0,\, d_{\mr{std}} f\right)\ \ \ \in \Omega^1 \oplus \Omega^1.
\end{equation}
Denoting by $x_1,\ldots,x_n$ the standard coordinates on $\mbR^n$, any $α\in A^{p,q}$ is now in a unique way of the form
$$α = \sum_{I,J\subseteq \{1,\ldots,n\},\ |I| = p,\ |J| = q} φ_{I,J}(x_1,\ldots,x_n)d'x_I\wedge d''x_J$$
with $φ_{I,J}\in C^\infty$. The above $d'\colon C^\infty\to A^{1,0}$ and $d''\colon C^\infty\to A^{0,1}$ extend to $A$ in a unique way that satisfies the \emph{Leibniz rule}
\begin{equation}\label{eq:Leibniz_forms}
d(α\wedge β) = dα\wedge β + (-1)^{\deg α} α\wedge dβ,\ \ d\in \{d', d''\}.
\end{equation}
Concretely, this extension is given as
$$d(φd'x_I\wedge d''x_J) = \sum_{i = 1}^n \frac{\partial φ}{\partial x_i} dx_i\wedge d'x_I\wedge d''x_J,\ \ \ d\in \{d',d''\}.$$
The so-defined $d',d''\colon A\to A$ are bihomogeneous of bidegree $(1,0)$ resp. $(0,1)$.

Given an affine-linear map $f\colon \mbR^n\to \mbR^m$, there is a pull-back map $f^*\colon A^{p,q}(\mbR^m)\to A^{p,q}(\mbR^n)$ which stems from usual pull-back of differential forms. It commutes with $\wedge$, $d'$ and $d''$.

The integral of an $(n,n)$-form $η$ with compact support is defined as follows. Write $η = φd'x_1\wedge d''x_1\wedge \ldots\wedge d'x_n\wedge d''x_n$ and put
\begin{equation}\label{eq:def_integral_of_form}
\int_{\mbR^n} η := \int_{\mbR^n} φ
\end{equation}
where the right hand side is defined in terms of the Lebesgue integral for the standard volume on $\mbR^n$. It is immediate that, for an affine linear map $f\colon \mbR^n\to \mbR^n$,
\begin{equation}
\label{eq:integral_transformation_rule}
\int_{\mbR^n} f^*η = |\det f|\int_{\mbR^n} η.
\end{equation}
There is, in particular, no implicit choice of orientation involved. This also reflects in the fact that the forms $d'x_i\wedge d''x_i$ have degree $2$, hence pairwise commute, so the function $φ$ in \eqref{eq:def_integral_of_form} is independent of coordinate ordering.

Let $D = D(\mbR^n)$ denote the space of \emph{currents}, i.e. the topological dual of compactly supported forms $A_c$, cf. \cite{Lag}*{§1.1}. The topological aspect of the definition will never play a role in this paper. Write $D = \bigoplus_{p,q} D^{p,q}$ where $D^{p,q}$ is dual to $A^{n-p,n-q}_c$. There is an injective map $A^{p,q}\to D^{p,q},\ α\mapsto [α]$, defined by
$$[α](η) := \int_{\mbR^n} α\wedge η.$$
With the following sign conventions one defines derivatives $d'\colon D^{p,q}\to D^{p+1,q}$, $d''\colon D^{p,q}\to D^{p,q+1}$ as well as a product $\wedge\colon A^{p,q}\times D^{s,t} \to  D^{p+s, q+t}$:
\begin{equation}\label{eq:signs_for_currents}
\begin{aligned}
(d T)(η) & = (-1)^{\deg T + 1}T(d η),\ \ d\in \{d',d''\},\\
(α\wedge T)(η) &= (-1)^{\deg α \deg T}T(α\wedge η).
\end{aligned}
\end{equation}
Then it follows that, for homogeneous $α$, $β$, $T$ and $d\in \{d',d''\}$,
\begin{equation}\label{eq:forms_map_to_currents_nicely}
\begin{aligned}
(d[α])(η) &= (-1)^{\deg α + 1}\int α \wedge dη = \int dα \wedge η = [dα](η),\\
(α\wedge [β])(η) &= (-1)^{\deg α \deg β}\int β\wedge α \wedge η = \int α\wedge β \wedge η = [α\wedge β](η),
\end{aligned}
\end{equation}
for every compactly supported test form $η$, so the inclusion $A\to D$ commutes with $\wedge$, $d'$ and $d''$. Furthermore, the Leibniz rule extends:
\begin{equation}\label{eq:Leibniz_current}
d(α\wedge T) = dα\wedge T + (-1)^{\deg α} α\wedge dT,\ \ d\in \{d',d''\}.
\end{equation}
Let $f\colon \mbR^n \to  \mbR^m$ be an affine linear map. Then there is a push-forward map $D^{p,q}_c (\mbR^n) \to  D^{p + m -n, q +m -n}(\mbR^m)$ from currents with compact support. It is defined by
\begin{equation}\label{eq:push_for_currents}
(f_*T)(η) = T(f^*η).
\end{equation}
Since $f^*(dη) = d(f^*η)$ for $η\in A_c(\mbR^m)$ and $d\in \{d',d''\}$, one obtains $f_*(dT) = d(f_*T)$ by duality, cf. \eqref{eq:signs_for_currents}.

\subsection{Polyhedral currents}
\label{ss:polyhedral_currents}

By \emph{polyhedron} in $\mbR^n$ we mean a subset $σ$ that may be written as the intersection of finitely many (not necessarily rational) half-spaces. Denote by $N_σ$ the linear space spanned by all $x-y$, $x,y\in σ$ and by $M_σ := N_σ^\vee$ its $\mbR$-dual. The \emph{dimension} of $σ$ is the dimension of $N_σ$.

A \emph{polyhedral complex} is a locally finite set of polyhedra $\mcT$ which is stable under taking faces and is such that $σ_1\cap σ_2$ is a face of both $σ_1$ and $σ_2$ for every $σ_1,σ_2\in \mcT$. The $d$-dimensional (resp. $r$-codimensional in $\mbR^n$) polyhedra of a polyhedral complex are denoted by $\mcT_d$ (resp. $\mcT^r$).

Let $C^\infty(σ)$ denote the smooth functions on $σ$, i.e. all $φ\colon σ\to \mbR$ such that there is some smooth function $\wt {φ}\in C^\infty(\mbR^n)$ with $\wt{φ}\vert_σ = φ$. Smooth $(p,q)$-forms on $σ$ are defined by an analogous restriction process. Let $L_σ = x + N_σ,\ x\in σ,$ be the smallest affine linear space containing $σ$. There is a well-defined space of $(p,q)$-forms $A^{p,q}(L_σ)$ because $(p,q)$-forms transform naturally under affine linear maps. Then
$$A^{p,q}(σ) := C^\infty(σ)\tensor_{C^\infty(L_σ)}A^{p,q}(L_σ).$$
Equivalently, it is the space of smooth $(p,q)$-forms on the interior $σ^\circ$ of $σ$ in $L_σ$ that come by restriction from $A^{p,q}(L_σ)$. Note that $A^{p,q}(σ) = 0$ if $\dim σ < \max\{p,q\}$ and that there is a restriction map $A^{p,q}(σ)\to A^{p,q}(τ)$, $α\mapsto α\vert_τ$ for every inclusion of polyhedra $τ\subseteq σ$ which commutes with $\wedge$, $d'$ and $d''$. The following definition of weight is inspired by Chambert-Loir--Ducros' definition of a calibration, cf. \cite{CLD}*{§1.5}.

\begin{defn}\label{def:weight}
A \emph{weight} on a polyhedron $σ$ is a generator $µ\in \det N_σ$ up to sign. The convention for $0$-dimensional polyhedra here is that the determinant of the $0$-space is $\mbR$ itself and that a weight is a \emph{positive} scalar. The pair $[σ,µ]$ is called a \emph{weighted polyhedron}. A \emph{weighted polyhedral complex} is the datum of a polyhedral complex $\mcT$ together with weights $(µ_σ)_{σ\in \mcT}$ for all its polyhedra.
\end{defn}
Equivalently, a weight is the choice of a Haar measure for $N_σ$. We denote by $µ^\vee \in \det M_σ$ the dual of $µ$.
\begin{ex}\label{ex:weights}
Every \emph{rational} polyhedron $σ\subseteq \mbR^n$ has a natural weight with respect to the lattice $\mbZ^n\subseteq \mbR^n$. Namely $N_σ\cap \mbZ^n$ is a lattice in $N_σ$ and the choice of a generator $µ_0 \in \det_{\mbZ} (N_σ\cap \mbZ^n)$ is unique up to sign. Every other weight is in a unique way of the form $λ µ_0$, $λ>0$.
\end{ex}

Let $[σ,µ]$ be a weighted polyhedron of dimension $d$ and let $η\in A_c^{d,d}(σ)$. Pick any coordinates $x_1,\ldots,x_d\in M_σ$ such that $µ^\vee = x_1\wedge\ldots\wedge x_d$ and write $η = φd'x_1\wedge d''x_1 \wedge \ldots \wedge d'x_d\wedge d''x_d$. (The $x_i$ are defined up to translation on $L_σ$, so their differentials $d'x_i$ and $d''x_i$ are canonical.) Then set
\begin{equation}
\int_{[σ,µ]} η := \int_{σ} φ
\end{equation}
where the right hand side is the Lebesgue integral with respect to the volume defined by the choice of isomorphism $(x_1,\ldots,x_d)\colon N_σ\iso \mbR^d$. The transformation rule \eqref{eq:integral_transformation_rule} ensures that this is well-defined. In this way, $[σ,µ]$ is viewed as element of $D^{r,r}$, where $r = n-d$ is the codimension of $σ$.

The following definitions are due to Gubler--Künnemann, cf. \cite{GK}*{Def. 2.3}. A \emph{polyhedral current} is a current that is a locally finite sum of currents of the form $α\wedge [σ,µ]$. Deviating from their notation, we write $P \subseteq D$ for the space of all polyhedral currents and $P^{p,q,r}\subseteq D^{p+r, q+r}$ for those which are locally finite sums of $α\wedge [σ,µ]$ with $σ$ of codimension $r$ and $α\in A^{p,q}(σ)$. One easily checks the direct sum decomposition $P = \bigoplus_{p,q,r} P^{p,q,r}.$ We also say that elements of $P^{p,q,r}$ are trihomogeneous of tridegree $(p,q,r)$.
\begin{rmk}\label{rmk:presentation_of_current}
When presenting a polyhedral current $T$ as a locally finite sum $T = \sum_{i\in I} α_i\wedge [σ_i,µ_i]$, the datum of all $(α_i,σ_i,µ_i)_{i\in I}$ is unique up to locally finitely many operations of the following kinds: Subdividing the $σ_i$, adding/removing terms with $α=0$, replacing $(α,σ,µ)$ by $(λα,σ,λ^{-1}µ)$ for some $λ>0$, and exchanging $(α_1,σ,µ) + (α_2,σ,µ)$ and $(α_1+α_2,σ,µ)$.
\end{rmk}
\begin{defn}\label{def:polyhedral_derivative}
Let $T$ be a polyhedral current, say $T = \sum_{i\in I} α_i\wedge [σ_i,µ_i]$. Its \emph{polyhedral derivatives} are the polyhedral currents
$$d'_PT := \sum_{i\in I} (d'α_i)\wedge [σ_i,µ_i],\ \ \ d''_PT := \sum_{i\in I} (d''α_i)\wedge [σ_i, µ_i].$$
\end{defn}
It has been remarked before, cf. \cite{GK}*{Rmk. 2.4 (iii)}, that $d'_PT$ and $d'T$ resp. $d''_PT$ and $d''T$ need not coincide. The derivatives $d'T$ and $d''T$ may even be non-polyhedral, cf. Ex. \ref{ex:deriv_of_polyhed_not_polyhed} below.

A polyhedral complex $\mcT$ is \emph{subordinate} to $T$ if there is a presentation of the form $T = \sum_{σ\in \mcT} α_σ\wedge [σ,µ_σ]$. With such $\mcT$ fixed, the $α_σ$ and $µ_σ$ are uniquely determined up to the replacement of $(α_σ,µ_σ)$ by $(λα_σ,λ^{-1}µ_σ)$, where $λ>0$.

\subsection{Functoriality}
\label{ss:functoriality}
For an exact sequence of finite dimensional $\mbR$-vector spaces
$$0\to N_1\to N_2\to N_3\to 0,$$
there is a canonical-up-to-sign isomorphism $\det N_2 = \det N_1\tensor \det N_3$. So given weights $µ_i$ for $N_i$ for two out of $\{N_1,N_2,N_3\}$, they uniquely determine a weight for the third space through the relation
\begin{equation}
\label{eq:prod_of_weight}
µ_2 = µ_1\wedge µ_3 := µ_1\wedge \wt{µ_3}
\end{equation}
where $\wt{µ_3}\in \bigwedge^{\dim N_3}N_2$ is any lift of $µ_3$.

There is a space $PS(σ)$ of \emph{piecewise smooth forms} on a polyhedron $σ$. By definition, a piecewise smooth form is the datum of a polyhedral complex $\mcT$ with $σ = \cup_{ρ\in \mcT} ρ$ and smooth forms $α_ρ\in A(ρ)$, $ρ\in \mcT$, such that $α_ρ\vert_τ = α_τ$ for all $τ\subseteq ρ$; up to subdivision. We write $PS^{p,q}(σ)$ for those with all $α_ρ$ of bidegree $(p,q)$. If $µ$ is a weight for $σ$ and $α = (α_ρ)_{ρ\in \mcT}\in PS(σ)$ as before, we define the polyhedral current
\begin{equation}\label{eq:pws_as_current}
α\wedge [σ,µ] = \sum_{ρ\in \mcT,\ \dim ρ\ =\ \dim σ} α_ρ\wedge [ρ, µ].
\end{equation}
Here $µ$ defines a weight for $ρ$ because $N_ρ = N_σ$ for dimension reasons. For fixed $µ$, this defines an embedding $PS(σ) \subseteq D(\mbR^n)$.

Let $f\colon \mbR^n\to \mbR^m$ be an affine linear map and $σ\subseteq \mbR^n$ a polyhedron. Then $f(σ)$ is again a polyhedron. Let $µ$ be a weight on $σ$ and $ν$ a weight on $f(σ)$. Then $K := \ker\big((f-f(0))\vert_{N_σ}\colon  N_σ \to  N_{f(σ)}\big)$ acquires a canonical weight $δ$ through \eqref{eq:prod_of_weight} and there is a natural fiber integration map for forms with compact support, $f_{δ,*}\colon A_c^{p,q}(σ) \to  PS^{p-k,q-k}(f(σ))$, where $k = \dim K$.
It satisfies the projection formula
$$\int_{[σ,µ]} α\wedge f^*η = \int_{[f(σ), ν]} (f_{δ,*}α) \wedge η$$
which determines it uniquely. In other words, fiber integration provides a representative for the push-forward from \eqref{eq:push_for_currents},
\begin{equation}\label{eq:push_forward_fiber_integration}
f_*(α\wedge [σ,µ]) = (f_{δ,*} α)\wedge [f(σ), ν].
\end{equation}
In particular, the push-forward of a polyhedral current (with relatively compact support) is polyhedral again. Note that if $α\wedge [σ,µ]\in P_c^{p,q,r}$ and $\dim K = k$ as before, then $f_*(α\wedge [σ,µ]) \in P^{p-k,q-k,r + m - n + k}$. So $f_*$ is not trihomogeneous, but only bihomogeneous.

\begin{ex}
If $T = φ\wedge [σ,µ]\in P^{0,0,r}$ is a weighted polyhedron with smooth coefficient (cf. \cite{GK}*{§1}), then $f_*T \neq 0$ only if $f\vert_σ$ is injective. In this case $f_*(φ\wedge [σ,µ]) = (φ\circ f^{-1}) \wedge [f(σ), f(µ)]$, which is precisely the classical push-forward of weighted polyhedra in tropical geometry that underlies e.g. the Sturmfels--Tevelev multiplicity formula \cite{ST}.
\end{ex}

Given a \emph{surjective} affine linear map $f\colon \mbR^n\to \mbR^m$ and a current $T$ on $\mbR^m$, we may now also define a \emph{pull-back current} $f^*T \in D(\mbR^n)$. Namely the fiber integral $f_*η$ of a smooth form $η\in A_c^{p,q}(\mbR^n)$ (with respect to the standard weights on $\mbR^n$ and $\mbR^m$) is again smooth and we put
$$(f^*T)(η) := T(f_*η),\ \ \ η\in A_c(\mbR^n).$$
If $T = α\wedge [σ,µ]$, then one easily finds $f^*T = f^*α\wedge [f^{-1}σ, ν]$, where $ν = δ\wedge μ$ for the natural weight $δ$ on $\ker (f-f(0))$. Since $f_*(dη) = d(f_*η)$ for $η\in A_c(\mbR^n)$ and $d\in \{d',d''\}$, it follows by duality that $f^*(dT) = d(f^*T)$ for any current $T$.

\begin{ex}
Assume $f\colon \mbR^n\to \mbR^n$ is bijective and affine linear. Let $μ$ be the standard weight on $\mbR^n$. Then $f(μ) = |\det f|\, μ$. The fiber weight $δ$ on $\ker(f - f(0)) = \{0\}$ is thus $|\det f|$ and fiber integration is given by
\begin{equation}
f_{δ,*}(α) = |\det f| \, f^{-1,*}(α).
\end{equation}
The transformation rule \eqref{eq:integral_transformation_rule} implies that this satisfies the projection formula:
$$\int_{[\mbR^n, μ]} α\wedge f^*η = \int_{[\mbR^n, μ]} f^*(f^{-1,*}(α) \wedge η) \overset{\eqref{eq:integral_transformation_rule}}{=} |\det f| \int_{[\mbR^n, μ]} f^{-1,*}(α) \wedge η = \int_{[\mbR^n,μ]} f_{δ,*}(α)\wedge η.$$
Interchanging the roles of $α$ and $η$, the equality of leftmost and rightmost term shows
\begin{equation}
f^*(α\wedge [\mbR^n,μ]) = f^*(α)\wedge [\mbR^n,μ].
\end{equation}
\end{ex}

\subsection{Stokes' Theorem}
\label{ss:stokes}
By its very definition, a $(p,q)$-form $α$ on $\mbR^n$ may be viewed as an alternating form in $p+q$ variables on $\mbR^n\times \mbR^n$ with values in $C^\infty$. Given $w\in \mbR^n\times \mbR^n$, the \emph{contraction} $(α,w)$ of $α$ with $w$ (interior derivative) is defined as the multilinear form resulting from inserting and fixing $w$ as the first entry of $α$. This operation is characterized by the Leibniz rule
\begin{equation}
(α\wedge β, w) = (α,w)\wedge β + (-1)^{\deg α} α\wedge (β,w)
\end{equation}
and the identities
\begin{equation}
(d'φ,w) = \frac{\partial φ}{\partial w_1},\ \ \ (d''φ, w) = \frac{\partial φ}{\partial w_2}, \ \ \ φ\in C^\infty,\ w = (w_1,w_2).
\end{equation}
Recall that a \emph{facet} of a polyhedron is a face of codimension $1$.
\begin{defn}\label{def:normal_vector}
Let $[σ,µ]$ be a weighted polyhedron and $τ\subset σ$ a facet that is endowed with a weight $ν$. Then there is a unique vector $\ob{n_{σ,τ}}\in N_σ/N_τ$ that points in direction of $σ$ and is such that $µ = ν \wedge \ob{n_{σ,τ}}$ in the sense of \eqref{eq:prod_of_weight}. A \emph{normal vector} for $τ\subset σ$ is any choice of lift $n_{σ,τ}\in N_σ$.
\end{defn}
Assume $m = \dim σ$. The (first) \emph{boundary integral} of $α\in A_c^{m-1,m}(σ)$ over $τ$ is defined as
\begin{equation}\label{eq:def_boundary_integral_facet}
\int_{\partial'_τ [σ,µ]} α := - \int_{[τ,ν]} (α, n''_{σ,τ})\vert_{τ}.
\end{equation}
The convention in notation here is $v' = (v,0)$ and $v'' = (0,v)$ for any vector $v\in \mbR^n$. The restriction $(α,n''_{σ,τ})\vert_τ$ is independent of the choice of normal vector and the whole expression is independent of the choice of $ν$. Define the (first) \emph{boundary integral of $σ$} as
\begin{equation}\label{eq:def_boundary_integral_smooth}
\int_{\partial' [σ,µ]} α := \sum_{τ\subset σ\ \text{a facet}}\int_{\partial'_τ [σ,µ]} α.
\end{equation}
The definition of the (second) boundary integral differs by a sign which is motivated by Ex. \ref{ex:sign_Stokes} below. For $β\in A_c^{m,m-1}(σ)$, put
\begin{equation}\label{eq:def_boundary_integral_facet_second}
\int_{\partial''_τ [σ,µ]} β := \int_{[τ,ν]} (α, n'_{σ,τ})\vert_{τ},\ \ \ \int_{\partial'' [σ, μ]} β = \sum_{τ\subset σ\ \text{a facet}} \int_{\partial''_τ [σ, μ]} β.
\end{equation}
\begin{prop}[\cite{CLD}*{Lemme 1.5.7}, Stokes' Theorem]\label{prop:stokes_polyhedron}
Let $[σ,µ]$ be an $m$-dimensional weighted polyhedron and let $α\in A^{m-1,m}_c(σ)$ and $β\in A^{m,m-1}_c(σ)$. Then
$$\int_{[σ,µ]}d'α = \int_{\partial' [σ,µ]} α,\ \ \ \int_{[σ,µ]} d''β =  \int_{\partial'' [σ,µ]} β.$$
\end{prop}

\begin{ex}\label{ex:sign_Stokes}
Prop. \ref{prop:stokes_polyhedron} is essentially just the following statement. For every smooth function $ρ\colon [0,1] \to \mbR$,
$$\int_0^1 ρ'(x)d'x\wedge d''x = ρ(1) - ρ(0) = - \int_0^1 ρ'(x) d''x\wedge d'x.$$
The differing signs explain the sign change from \eqref{eq:def_boundary_integral_facet} to \eqref{eq:def_boundary_integral_facet_second}.
\end{ex}

\begin{ex}\label{ex:deriv_of_polyhed_not_polyhed}
Prop. \ref{prop:stokes_polyhedron} says $d'[σ,µ] = -\partial'[σ,µ]$ and $d''[σ,µ] = -\partial''[σ,µ]$ as currents, but these derivatives are never polyhedral if $\dim σ > 0$. (The difference in sign with Stokes' Theorem comes from \eqref{eq:signs_for_currents}.) Namely they have support on the union of facets $\partial σ$ of $σ$. If $\dim σ = m$, then $\partial σ$ is an $(m-1)$-dimensional polyhedral set, so $η\vert_{\partial σ} = 0$ for every $η\in A^{m-1,m}(σ)$ resp. $η\in A^{m,m-1}(σ)$, but not necessarily
$$\int_{\partial' [σ,µ]}η = 0\ \ \ \text{resp.}\ \ \ \int_{\partial'' [σ,µ]}η = 0.$$
\end{ex}

\section{$δ$-Forms}
\label{s:delta_forms}
We consider forms, currents and polyhedra on $\mbR^n$ in the following.
\begin{defn}\label{def:delta-form}
A \emph{$δ$-form} is a polyhedral current $T$ such that both $d'T$ and $d''T$ are again polyhedral.
\end{defn}
This definition turns out to be equivalent to the familiar concept of balancing for $T$.
\begin{defn}\label{def:balance}
Let $\mcT$ be a polyhedral complex, $µ_σ$, $σ\in \mcT$, a family of weights for its polyhedra and $α_σ\in A(σ)$, $σ\in \mcT$ a family of smooth forms. This datum is called \emph{balanced}, if the following two equivalent conditions are met.

(1) For all polyhedra $τ\in \mcT$, the sum
\begin{equation}\label{eq:def_balancing}
\sum_{σ\in \mcT,\ τ\subset σ\ \text{a facet}} α_σ\vert_τ \tensor n_{σ,τ} \in A(τ)\tensor_{\mbR} \mbR^n.
\end{equation}
lies in the subspace $A(τ)\tensor_{\mbR}N_τ$. The normal vectors $n_{σ,τ}$ here are taken for the weights $µ_σ$ and $µ_τ$.

(2) For every polyhedron $τ\in \mcT$, every affine linear function $z$ with \emph{constant} restriction $z\vert_τ$ and normal vectors $n_{σ,τ}$ as before,
\begin{equation}\label{eq:def_balancing_lin_functions}
\sum_{σ\in \mcT,\ τ\subset σ\ \text{a facet}} \frac{\partial z}{\partial n_{σ,τ}} α_σ \vert_τ = 0.
\end{equation}
Since $z\vert_τ$ is constant, this expression does not depend on the choices of the $n_{σ,τ}$.
\end{defn}

\begin{proof}[Proof of the equivalence of (1) and (2).] Assume $z\vert_τ$ to be constant and consider the pairing
$$A(\mbR^n)\tensor_{\mbR}\mbR^n \longrightarrow  A(τ),\ α\tensor v \longmapsto (d'z\wedge α, v')\vert_τ.$$
It follows from the Leibniz rule that 
$$(d'z\wedge α, v')\vert_τ = (\partial z/\partial v)\cdot α\vert_τ,$$
so the pairing factors through $A(τ)\tensor_{\mbR} \mbR^n$ and is simply the $A(τ)$-linear extension of $v\mapsto \partial z/\partial v$. The proof is now the statement that a vector $v$ lies in $N_τ$ if and only if $\partial z/\partial v = 0$ for every affine linear function $z$ with constant restriction $z\vert_τ$.
\end{proof}

Formulation (1) is closer to the usual condition of balancing in tropical geometry but makes implicit use of the existence of the ambient space $\mbR^n$. Formulation (2) in turn is more suitable for generalizations to abstract polyhedral complexes, cf. \cite{Mih_delta_nonarch}.

Being balanced is stable under the four operations in Rmk. \ref{rmk:presentation_of_current}, so only depends on the current $T = \sum_{σ\in \mcT} α_σ\wedge [σ,µ_σ]$. Also note that \eqref{eq:def_balancing} and \eqref{eq:def_balancing_lin_functions} are trihomogeneous in $α$ and that only polyhedra of a fixed dimension occur. One obtains that $T = \sum_{p,q,r} T^{p,q,r}$, with $T^{p,q,r}$ of tridegree $(p,q,r)$, is balanced if and only if each $T^{p,q,r}$ is. 

\begin{thm}\label{thm:delta_forms}
A polyhedral current $T$ is a $δ$-form if and only if it is balanced. In particular, it is a $δ$-form if and only if $T^{p,q,r}$ is a $δ$-form for all $(p,q,r)$.

Furthermore, $T$ is already a $δ$-form if one out of $d'T$, $d''T$ is again polyhedral.
\end{thm}

\begin{proof}
(1) We first assume that $T$ is of tridegree $(p,q,r)$. Let $\mcT$ be a weighted polyhedral complex subordinate to $T$, say
$$T = \sum_{σ\in \mcT^r}α_σ\wedge [σ,µ_σ],\ \ \ α_σ\in A^{p,q}(σ),$$
and let $η\in A^{n-p-r,n-q-r}_c$ be a test form. One obtains from the Leibniz rule and Stokes' Theorem that
\begin{align}
(d'T - d'_PT)(η) & = \sum_{τ\in \mcT^{r+1}}\int_{[τ,µ_τ]}\sum_{σ\in \mcT^r,\ τ\subset σ} (α_σ\wedge η, n''_{σ,τ})\vert_τ \label{eq:boundary_general}\\
& = \sum_{τ} \left[\sum_{τ\subset σ} \big((α_σ,n''_{σ,τ}) \wedge [τ,µ_τ]\big)(η) + (-1)^{\deg α} \int_{[τ,µ_τ]} \sum_{τ\subset σ} α_σ \wedge (η,n''_{σ,τ})\vert_τ\right]. \nonumber
\end{align}
The individual contractions $(α_σ,n''_{σ,τ})$ and $(η,n''_{σ,τ})$ depend on the choices of normal vectors, but the total expression does not. We henceforth fix the choices $n_{σ,τ}$. The terms $\big((α_σ,n''_{σ,τ}) \wedge [τ,µ_τ]\big)(η)$ always define polyhedral currents. So the statement to prove is that $T$ is balanced if and only if the following is a polyhedral current,
$$η\longmapsto (-1)^{\deg α}\sum_{τ} \int_{[τ,µ_τ]} \sum_{τ\subset σ} α_σ \wedge (η,n''_{σ,τ})\vert_τ.$$

(2) Assume first that $T$ is balanced, fix some $τ$ and write
\begin{equation}\label{eq:aux_1}
\sum_{τ\subset σ} α_σ\vert_τ \tensor n_{σ,τ} = \sum_{i \in I} β_i \tensor v_i,\ β_i\in A^{p,q}(τ),\ v_i\in N_τ,
\end{equation}
according to \eqref{eq:def_balancing}. Then
\begin{equation}\label{eq:aux_2}
\sum_{τ\subset σ} α_σ \wedge (η,n''_{σ,τ})\vert_τ = \sum_{i \in I} β_i\wedge (η,v''_i)\vert_τ.
\end{equation}
By the Leibniz rule,
\begin{equation}\label{eq:aux_3}
β_i\wedge (η, v''_i) = (-1)^{\deg β}(β_i\wedge η, v''_i) + (-1)^{\deg β +1} (β_i,v''_i)\wedge η.
\end{equation}
Since $β_i\wedge η$ is of bidegree $(\dim τ, \dim τ + 1)$, the first summand vanishes. The (integral over $[τ,μ_τ]$ of the) second summand defines a polyhedral current in $η$. Taking the sum over $i$ and $τ$ shows that $d'T$ is a polyhedral current. The same argument works for $d''T$, proving that a trihomogeneous balanced polyhedral current is a $δ$-form. 

(3) Conversely assume that $T$ is \emph{not} balanced, our claim being that $d'T$ is \emph{not} polyhedral. (We still assume that $T$ has tridegree $(p,q,r)$ currently.) Generally, if $S$ is a polyhedral current, $C$ some polyhedral sets with $\Supp S \subseteq C$ and $η\in A_c$ a test form, then $η\vert_C = 0$ implies $S(η) = 0$. In the situation at hand, we have already seen that $\Supp (d'T-d'_PT)$ is contained in the codimension $r+1$ skeleton $\bigcup_{τ\in \mcT^{r+1}} τ$ and our approach is to construct a test form $η$ with $η\vert_τ = 0$ for all $τ$ but $(d'T-d'_PT)(η)\neq 0$.
Pick $τ$ and $z$ such that \eqref{eq:def_balancing_lin_functions} is not satisfied, i.e. $z$ is an affine linear function with constant restriction $z\vert_τ$ and such that
\begin{equation}\label{eq:beta}
β := \sum_{τ\subset σ} \frac{\partial z}{\partial n_{σ,τ}} α_σ\vert_τ\neq 0.
\end{equation}
There exists a bump test form $\ob{η} \in A_c^{\dim τ - p, \dim τ - q}$ with the properties that $\Supp \ob{η}\cap τ' \neq \emptyset$, $τ'\in \mcT^{r+1}$, only for $τ' = τ$ and
$$\int_{[τ,µ_τ]} β\wedge \ob{η} \neq 0.$$
Then the $τ$-contribution to \eqref{eq:boundary_general} for the test form $η = (-1)^{\deg α} d''z \wedge \ob{η}$ is simply
$$\sum_{τ\subset σ} (d''z\wedge α_σ\wedge \ob{η}, n''_{σ,τ})\vert_τ = β \wedge \ob{η}.$$
Here we combined the Leibniz rule for $(\ \ \ ,n''_{σ,τ})$ with the properties $d''z\vert_τ = 0$ and $(d''z, n''_{σ,τ}) = \partial z/\partial n_{σ,τ}$. Thus $(d'T - d'_PT)(d''z\wedge \ob{η}) \neq 0$ even though $d''z\wedge \ob{η}\vert_{τ'} = 0$ for every $τ'\in \mcT^{r+1}$. So $d'T$ cannot be polyhedral and hence $T$ is not a $δ$-form. Note that arguments (2) and (3) show the stronger statement that a trihomogeneous $T$ is balanced if and only if \emph{one out of} $d'T$ and $d''T$ is polyhedral, i.e. they prove the last statement of Thm. \ref{thm:delta_forms} for trihomogeneous $T$.

(4) Now consider a general $δ$-form $T = \sum_{p,q,r} T^{p,q,r}$ with $T^{p,q,r}$ of the indicated tridegree. Our claim is that each $T^{p,q,r}$ is a $δ$-form.
Since $T^{p_1,q_1,r}$ and $T^{p_2,q_2,r}$ lie in different bidegrees as currents for $(p_1,q_1)\neq (p_2,q_2)$ and since $d'$ and $d''$ are bihomogeneous, it is enough to prove that all $T^r := \sum_{p,q}T^{p,q,r}$ are $δ$-forms.
Polyhedral currents may be added ad libitum, so it is sufficient to show that all $(d'T^r - d'_PT^r)$ and $(d''T^r - d''_PT^r)$ are polyhedral. Assume for the sake of contradiction that there is some $r_0$ with $d'T^{r_0}$ not polyhedral and assume that $r_0$ is chosen minimal. Then the previous arguments imply that there is some point
$$x\in \Supp (d'T^{r_0} - d'_PT^{r_0}) \setminus \Supp \left(\sum_{r > r_0} d'T^r - d'_PT^r\right)$$
that has no open neighborhood $x\in U$ such that $d'T^{r_0}\vert_U$ is polyhedral. (Take $x\in \Supp β$ where $β$ is as in \eqref{eq:beta}.) Using minimality of $r_0$, we conclude that $d'T$ cannot be polyhedral. The same argument applies with $d''$ instead of $d'$. Thus we obtain that $T$ is a $δ$-form, if and only if each $T^{p,q,r}$ is a $δ$-form, if and only if each $T^{p,q,r}$ is balanced, if and only if $T$ is balanced.
\end{proof}

\begin{defn}
(1) We denote by $B^{p,q,r}=B^{p,q,r}(\mbR^n)$ the space of $δ$-forms of the indicated tridegree and by $B = \bigoplus_{p,q,r} B^{p,q,r}$ the space of all $δ$-forms. Write $B^{p,q} = \bigoplus_{r} B^{p-r,q-r,r}$ for the space of $δ$-forms of bidegree $(p,q)$ in the sense of currents.

(2) Since $d'd'T = 0$, Thm. \ref{thm:delta_forms} implies that $d'T$ is a $δ$-form. Similarly for $d''$, so one obtains derivatives
$$d'\colon B^{p,q}\longrightarrow B^{p+1,q},\ \ \ d''\colon B^{p,q}\longrightarrow B^{p,q+1}.$$
The balancing condition \eqref{eq:def_balancing} is stable under $d'_P$ and $d''_P$, so the polyhedral derivatives restrict to operators
$$d'_P\colon B^{p,q,r} \longrightarrow  B^{p+1,q,r},\ \ \ d''_P\colon B^{p,q,r}\longrightarrow B^{p,q+1,r}.$$
Define the boundary operators $\partial' := d'_P - d'$ and $\partial'' := d''_P - d''$. It will be explained below, cf. \eqref{eq:boundary_formula_1}, that these are trihomogeneous in the sense
$$\partial'\colon  B^{p,q,r} \longrightarrow  B^{p,q-1,r+1},\ \ \ \partial''\colon B^{p,q,r}\longrightarrow B^{p-1,q,r+1}.$$
\end{defn}

\begin{lem}\label{lem:differential_calculus_delta_forms}
\begin{enumerate}[wide, labelindent=0pt, labelwidth=!, label=(\arabic*), topsep=2pt, itemsep=2pt]
\item The boundary derivatives satisfy
\begin{align*}
0 &= \partial'\partial' = \partial''\partial'',\\
0 &= \partial'\partial'' + \partial''\partial',\\
0 &= \partial'd'_P + d'_P \partial' = \partial''d''_P + d''_P \partial'',\\
0 &= \partial' d''_P + d'_P \partial'' + \partial'' d'_P + d''_P \partial'.
\end{align*}

\item Assume that $T\in B_c(\mbR^n)$ has compact support and that $f\colon \mbR^n\to \mbR^m$ is an affine linear map. Then $f_*T\in B_c(\mbR^m)$ is also a $δ$-form.

\item Let $f\colon \mbR^n\to \mbR^m$ be a surjective affine linear map and $S\in B(\mbR^m)$. Then $f^*S\in B(\mbR^n)$ is also a $δ$-form.
\end{enumerate}
\end{lem}
\begin{proof}
(1) The necessary observation is that $d'_P$, $\partial'$, $d''_P$ and $\partial''$ are all trihomogeneous of different tridegrees. The stated relations then follow from the identities $(d')^2 = (d'')^2 = 0$ and $d'd'' = -d''d'$.

(2) and (3) follow from the fact that $f_*$ and $f^*$ commute with $d'$ and $d''$ and preserve the property of being polyhedral.
\end{proof}

\begin{ex}\label{ex:delta_form_by_degcree}
Every polyhedral current $T$ of tridegree $(n-r,q,r)$ or $(p,n-r,r)$ is a $δ$-form. This follows from the observation that then all terms $α_σ\vert_τ$ in \eqref{eq:def_balancing} vanish. Alternatively, one argues that $d'T = 0$ resp. $d''T = 0$ because $T$ is of bidegree $(n,q+r)$ resp. $(p+r, n)$ as current and applies Thm. \ref{thm:delta_forms}.
\end{ex}

\begin{lem}\label{lem:pws_as_delta_forms}
The $δ$-forms $B^{p,q,0}(\mbR^n)$ are precisely the currents of the form $α\wedge [\mbR^n,µ_{\mr{std}}]$ for a piecewise smooth $(p,q)$-form $α\in PS^{p,q}(\mbR^n)$.
\end{lem}
\begin{proof}
Assume that $T\in P^{p,q,0}$ and write $T = \sum_{σ\in \mcT^0} α_σ\wedge [σ,µ_σ]$ for a weighted subordinate polyhedral complex $\mcT$. Observe that one may assume all occurring $µ_σ = µ_{\mr{std}}$ since $N_σ = \mbR^n$ for every $n$-dimensional $σ$. Any $τ\in \mcT^1$ is then a facet of precisely two $σ_1,σ_2\in \mcT^0$. One may choose $n_{σ_2,τ} = - n_{σ_1,τ}$. The balancing condition \eqref{eq:def_balancing} is then equivalent to $α_{σ_1}\vert_τ = α_{σ_2}\vert_τ$ for all such $τ\subset σ_1,σ_2$, which is equivalent to the $(α_σ)_{σ\in \mcT^0}$ defining a piecewise smooth form.
\end{proof}

\begin{ex}\label{ex:pws_delta_product}
Let $α\in PS(\mbR^n)$ be piecewise smooth and $T\in P(\mbR^n)$ a polyhedral current. Let $\mcT$ be a polyhedral complex that is subordinate to both $T$ and $α$, say
$$α = (α_σ)_{σ\in \mcT^0},\ \ \ T = \sum_{ρ\in \mcT} β_ρ\wedge [ρ,µ_ρ].$$
Define their product as
\begin{equation}\label{eq:prod_pws_polyhedral}
αT := \sum_{ρ\in \mcT} α\vert_ρ \wedge β_ρ\wedge [ρ,µ_ρ].
\end{equation}
The restriction $α\vert_ρ$ here is well-defined by the piecewise smooth property. If $T$ is a $δ$-form, then $αT$ is also a $δ$-form since the balancing condition \eqref{eq:def_balancing} is $PS$-linear. For example, the $δ$-preforms from \cite{GK}*{§2} are precisely the sums of products $α T$ where $α\in A(\mbR^n)$ is smooth and $T\in B^{0,0,r}$ a tropical cycle.
\end{ex}

We end this section by providing three ways to compute $\partial'T$. The case of $\partial''$ is the same by symmetry; it merely requires paying attention to difference in signs of \eqref{eq:def_boundary_integral_facet} and \eqref{eq:def_boundary_integral_facet_second}. Throughout, we assume that $T\in B^{p,q,r}$, say $T = \sum_{σ\in \mcT^r} α_σ\wedge [σ,µ_σ]$ for a subordinate weighted polyhedral complex $\mcT$. The proof of Thm. \ref{thm:delta_forms} shows that $\partial'T = \sum_{τ\in \mcT^{r+1}}β_τ\wedge [τ,µ_τ]$ for certain $β_τ$ which we would like to determine.

(1) Implicit in the proof of Thm. \ref{thm:delta_forms} is the following formula. Fix $τ$ and write
$$\sum_{τ\subset σ\ \text{a facet}} α_σ\vert_τ\tensor n_{σ,τ} = \sum_{i \in I} β_i\tensor v_i,\ \ \ β_i\in A(τ),\ v_i \in N_τ$$
as in \eqref{eq:aux_1}. Then \eqref{eq:boundary_general}, together with \eqref{eq:aux_2} and \eqref{eq:aux_3}, implies
\begin{equation}\label{eq:boundary_formula_1}
β_τ =  \sum_{i \in I} (β_i, v''_i) - \sum_{τ\subset σ} (α_σ, n''_{σ,τ})\vert_τ.
\end{equation}

(2) The next formula for $β_τ$ is more in line with formulation \eqref{eq:def_balancing_lin_functions} of the balancing condition. We use the definition $\partial' := d'_P - d'$ for all polyhedral currents in the following. Pick any affine linear map $f\colon \mbR^n\to \mbR^{\dim τ + 1}$ such that $f\vert_σ$ is injective for every $τ\subset σ\in \mcT^r$. Let $C = \bigcup_{τ\subset σ}σ$ be the polyhedral set formed by all $σ\in \mcT^r$ containing $τ$. Denote by $Z$ its boundary in the topological space $\bigcup_{σ\in \mcT^r} σ$.
The current $S = \sum_{τ\subset σ} α_σ\wedge [σ,µ_σ]$ has support contained in $C$ and is a $δ$-form away from $Z$. Since $f\vert_C$ has finite fibers, $\Supp S$ is relatively compact over $\mbR^m$, so the push-forward $f_*S$ is defined. It is a $δ$-form away from $f(Z)$. Moreover $f_*(d'_PS) = d'_P(f_*S)$ because $f\vert_S$ has finite fibers. It follows that $f_*(\partial'S) = \partial'f_*(S)$. Writing $\partial' f_*(S) = γ\wedge [f(τ),f(µ_τ)]$ away from $f(Z)$ shows
$$β_τ = f^*γ.$$
Now note that $(f_*S)\vert_{\mbR^{\dim τ + 1}\setminus f(Z)}$ lies in $B^{p,q,0}(\mbR^{\dim τ + 1}\setminus f(Z))$, i.e. $f_*S$ is given by a piecewise smooth form away from $f(Z)$ by Lem. \ref{lem:pws_as_delta_forms}. This makes the determination of $γ$ very simple: $f_*S$ is described near $f(τ)\setminus f(Z)$ by smooth forms $ω_i \in A^{p,q}(ρ_i)$ on two $(\dim τ + 1)$-dimensional polyhedra $ρ_1,\ ρ_2$ with $ρ_1\cap ρ_2 = f(τ)$. These satisfy $ω_1\vert_{f(τ)} = ω_2\vert_{f(τ)}$. Picking the normal vectors in the above \eqref{eq:boundary_formula_1} as $n := n_{ρ_1,f(τ)} = - n_{ρ_2,f(τ)}$ eliminates the first sum in \eqref{eq:boundary_formula_1} and shows
\begin{equation}\label{eq:boundary_formula_2}
γ = (ω_2, n'') - (ω_1,n'').
\end{equation}

(3) For the third and final formula, choose coordinate functions $x_1,\ldots,x_{n-r-1}\colon \mbR^n\to \mbR$ that restrict to a basis of $M_τ$. For each $σ$ containing $τ$, choose a non-constant affine linear function $z_σ\colon σ\to \mbR$ such that $z_σ\vert_τ$ is constant. (For example, one may choose an affine linear $z\colon \mbR^n\to \mbR$ such that $z\vert_τ$ is constant but $z_σ = z\vert_σ$ non-constant for every $σ\supset τ$.) Then every $α_σ$ can be uniquely expressed as
$$α_σ = α_σ^{(1)} + d'z_σ\wedge α_σ^{(2)} + d''z_σ \wedge α_σ^{(3)} + d'z_σ \wedge d''z_σ \wedge α_σ^{(4)}$$
with the $α_σ^{(j)}$ all $C^\infty(σ)$-linear combination of $d'x_I\wedge d''x_J$. Our claim is that
\begin{equation}\label{eq:boundary_formula_3}
β_τ = - \sum_{τ\subset σ\ \text{a facet}} \frac{\partial z_σ}{\partial n_{σ,τ}} α_σ^{(3)}\vert_τ.
\end{equation}
Note that already the individual summands are independent of the chosen $z_σ$.
\begin{proof}[Proof of the claim.]
In light of \eqref{eq:boundary_general}, we need to show that the following identity holds for all smooth forms $η\in A^{n-p-r-1, n-q-r}(\mbR^n)$ of complementary degree:
\begin{equation}\label{eq:boundary_formula_3_to_show}
\sum_{τ\subset σ\ \text{a facet}} (α_σ\wedge η, n_{σ,τ}'')\vert_τ = \sum_{τ\subset σ\ \text{a facet}} \frac{\partial z_σ}{\partial n_{σ,τ}} α_σ^{(3)}\wedge η\vert_τ.
\end{equation}
Since $d'z_σ\vert_τ = d''z_σ\vert_τ = 0$, it is immediately clear that
$$(α_σ\wedge η, n_{σ,τ}'')\vert_τ = (α_σ^{(1)} \wedge η, n_{σ,τ}'')\vert_τ + \frac{\partial z_σ}{\partial n_{σ,τ}} α_σ^{(3)}\wedge η\vert_τ.$$
Our task is thus to show
\begin{equation}\label{eq:boundary_formula_3_to_show_2}
\sum_{τ\subset σ\ \text{a facet}} (α_σ^{(1)}\wedge η, n_{σ,τ}'')\vert_τ = 0.
\end{equation}
Pick coordinate functions $y_1,\ldots,y_{r+1}\colon \mbR^n\to \mbR$ that extend $x_1,\ldots,x_{n-r-1}$ to a basis and that are constant along $τ$. Then $d'y_i\vert_τ = d''y_i\vert_τ = 0$. So if $η$ is of the form $d'y_i \wedge \wt{η}$, then already $(α_σ^{(1)}\wedge η, n_{σ,τ}'')\vert_τ = 0$. Similarly, if $η$ is a $C^\infty$-linear combination of monomials $d'x_I\wedge d''x_J$, then already $α^{(1)}_σ\wedge η = 0$ because this form is of degree $(n-r-1, n-r)$.

It is thus left to show \eqref{eq:boundary_formula_3_to_show_2} for forms $η = d''y_i \wedge \wt{η}$. We obtain that
\begin{equation}
\begin{aligned}
\sum_{τ\subset σ\ \text{a facet}} (α_σ^{(1)}\wedge d''y_i \wedge \wt{η}, n_{σ,τ}'')\vert_τ & = (-1)^{\deg T} \sum_{τ\subset σ\ \text{a facet}} \frac{\partial y_i}{\partial n_{σ,τ}} α_σ^{(1)} \wedge \wt{η}\vert_τ\\
& = (-1)^{\deg T} \sum_{τ\subset σ\ \text{a facet}} \frac{\partial y_i}{\partial n_{σ,τ}} α_σ \wedge \wt{η}\vert_τ.
\end{aligned}
\end{equation}
The last expression vanishes by the balancing condition \eqref{eq:def_balancing_lin_functions}.
\end{proof}

\section{Intersection Theory}
\label{s:intersection_theory}
\subsection{Main result}
\label{ss:intersection_main_results}

The definition of the $\wedge$-product of $δ$-forms is based on two specific constructions. The first is the product of piecewise smooth and $δ$-forms from Ex. \ref{ex:pws_delta_product}: By Lem. \ref{lem:pws_as_delta_forms}, every $δ$-form of tridegree $(p,q,0)$ is of the form $α\wedge [\mbR^n,μ_{\mr{std}}]$ for a (unique) piecewise smooth $(p,q)$-form $α$. We write $α$ by abuse of notation and define
\begin{equation}\label{eq:def_pws_product}
α\wedge T := αT,\ \ \ α\in B^{\bullet,\bullet,0},\ T\in B.
\end{equation}
The second construction is the exterior product of currents, cf. \cite{Demailly}*{§I.2}, defined as follows. Given homogeneous currents $T_1\in D(\mbR^n)$ and $T_2\in D(\mbR^m)$, it is the unique current $T_1\boxtimes T_2\in D(\mbR^n\times \mbR^m)$ such that
$$(T_1\boxtimes T_2)(p_1^*η_1\wedge p_2^*η_2) = (-1)^{\deg T_1\deg T_2}T_1(η_1)\cdot T_2(η_2).$$
In particular,
\begin{equation}\label{eq:tensor_product_current_derivative_d}
d(T_1\boxtimes T_2) = dT_1 \boxtimes T_2 + (-1)^{\deg T_1}T_1\boxtimes dT_2,\ \ \ d\in \{d',d''\}.
\end{equation}
The exterior product preserves polyhedral currents which follows from the identity
\begin{equation}
(α_1\wedge [σ_1,µ_1])\boxtimes (α_2\wedge [σ_2,µ_2]) = α_1\wedge α_2 \wedge [σ_1\times σ_2, µ_1\wedge µ_2].
\end{equation}
Relation \eqref{eq:tensor_product_current_derivative_d} then implies that the exterior product of $δ$-forms is a $δ$-form again. Moreover, one sees that if $T_i$ is of polyhedral tridegree $(p_i,q_i,r_i)$, then $T_1\boxtimes T_2$ has tridegree $(p_1+p_2,q_1+q_2,r_1+r_2)$. Separating \eqref{eq:tensor_product_current_derivative_d} by tridegree provides
\begin{equation}\label{eq:tensor_product_current_derivative_others}
\begin{aligned}
d_P(T_1\boxtimes T_2) &= d_PT_1 \boxtimes T_2 + (-1)^{\deg T_1}T_1\boxtimes d_PT_2,\ \ \ d_P\in \{d'_P,d''_P\},\\
\partial(T_1\boxtimes T_2) &= \partial T_1 \boxtimes T_2 + (-1)^{\deg T_1}T_1\boxtimes \partial T_2,\ \ \ \partial\in \{\partial',\partial''\}.
\end{aligned}
\end{equation}
We simply write $T_1\times T_2$ instead of $T_1\boxtimes T_2$ for $δ$-forms $T_1$ and $T_2$.

In the following, $\Delta = (\mr{id},\mr{id})_*[\mbR^n, μ_{\mr{std}}] \in B^{0,0,n}(\mbR^n\times \mbR^n)$ denotes the diagonal viewed as $δ$-form.
\begin{thm}\label{thm:main}
There is a unique way to define an associative product $\wedge\colon B\times B \to  B$ that satisfies the Leibniz rules with respect to $d'$ and $d''$, extends definition \eqref{eq:def_pws_product}, and can be computed by restriction to the diagonal, meaning
\begin{equation}\label{eq:inter_diagonal}
S\wedge T = p_{1,*}(\Delta \wedge (S\times T)).
\end{equation}
This product has the following additional properties.
\begin{enumerate}[wide, labelindent=0pt, labelwidth=!, label=(\arabic*), topsep=2pt, itemsep=2pt]
\item It is graded commutative and trihomogeneous in the sense $B^{p,q,r}\wedge B^{s,t,u} \subseteq B^{p+s, q+t, r+u}$. In particular, it satisfies the Leibniz rule with respect to the operators $\partial'$, $d_P'$, $\partial''$ and $d_P''$.

\item It commutes with pull-back: Given a surjective affine linear map $f\colon \mbR^n\to \mbR^m$ and $S,T\in B(\mbR^m)$, then $f^*(T\wedge S) = f^*T \wedge f^*S.$

\item It satisfies the projection formula: Given a surjective affine linear map $f\colon \mbR^n\to \mbR^m$, a $δ$-form $S\in B(\mbR^m)$ and a $δ$-form $T\in B(\mbR^n)$ with support compact over $\mbR^m$,
\begin{equation}\label{eq:projection_main_thm}
f_*(T\wedge f^*S) = f_*T \wedge S.
\end{equation}

\item It coincides with the tropical intersection products from \cites{AR, Esterov, GK} on $\bigoplus_{r} B^{0,0,r}$ whenever they are defined.
\end{enumerate}
\end{thm}

The idea of characterizing and constructing the tropical intersection product through divisor intersections and restriction to the diagonal is due to Allermann--Rau \cite{AR}.

\begin{proof}[Proof of the uniqueness assertion.]
If a $\wedge$-product exists as claimed, the Leibniz rule implies for piecewise smooth $α$ that
\begin{equation}\label{eq:prod_w_pws_aux}
d'd''(αT) = d'd''α\wedge T + α\wedge d'd''T + (-1)^{\deg α} (d'α\wedge d''T - d''α\wedge d'T). 
\end{equation}
In case of a piecewise linear function $φ$, the $δ$-forms $d'φ$ resp. $d''φ$ agree with $d'_Pφ$ resp. $d''_Pφ$ and are again piecewise smooth, because the contractions in \eqref{eq:boundary_formula_1} vanish for degree reasons. (This applies more generally to piecewise smooth functions.) It follows that
\begin{equation}\label{eq:div_inter}
d'd''φ\wedge T = d'd''(φT) - φ\wedge d'd''T + d''φ\wedge d'T - d'φ\wedge d''T
\end{equation}
is uniquely determined by the Leibniz rule and the piecewise smooth case. Denote by $x_1,\ldots,x_n$ and $y_1,\ldots,y_n$ the coordinate functions on $\mbR^n\times \mbR^n$ and let $φ_i := \max\{x_i,y_i\}$. Then, by \cite{AR}*{Rmk. 9.2}, the diagonal $\Delta$ is the product
$$\Delta = d'd''φ_1\wedge\ldots\wedge d'd''φ_n,$$
where the right hand side is a successive application of \eqref{eq:div_inter}. Again by \eqref{eq:div_inter} as well as the associativity of the $\wedge$-product, $Δ\wedge (S\times T)$ is now uniquely determined. Hence $S\wedge T = p_{1,*}(Δ\wedge (S \times T))$ is uniquely characterized by the stated conditions.
\end{proof}

The existence statement will be shown in the next section. Here, we give an application of Thm. \ref{thm:main} to the definition of a pull-back for all affine linear maps, not just surjective ones. It is specific to $δ$-forms, meaning it does not extend to polyhedral currents. Its construction is well-known for tropical cycles, cf. \cite{GK}*{Rmk. 1.4 (v)} for example.

\begin{propdef}\label{propdef:pullback}
Let $f\colon \mbR^n\to \mbR^m$ be an affine linear map and $S\in B(\mbR^m)$ a $δ$-form. There is a unique $δ$-form $f^*(S)\in B(\mbR^n)$, called the pull-back of $S$ along $f$, that satisfies the projection formula
\begin{equation}\label{eq:projection_formula_general}
f_*(T\wedge f^*S) = f_*T \wedge S.
\end{equation}
This pull-back is functorial in $f$ and commutes with $\wedge$-products as well as all the six differential operators.
\end{propdef}
\begin{proof}
Identity \eqref{eq:projection_formula_general} determines $f^*S$ uniquely because it determines all its values on test forms $η$ by
$$\begin{aligned}
(f^*S)(η) & = \int_{\mbR^n} f^*S \wedge η \\
& = \int_{\mbR^m} f_*(f^*S\wedge η) = \int_{\mbR^m} S \wedge f_*(η).
\end{aligned}$$
Just from this uniqueness, one may deduce all further properties. For example, for $d\in \{d',d''\}$ and for all $δ$-forms $T$,
$$\begin{aligned}
f_*(T\wedge d(f^*S)) & = (-1)^{\deg T} f_*\big(d(T\wedge f^*S) - dT\wedge f^*S\big)\\
& = (-1)^{\deg T} \big(d(f_*(T\wedge f^*S)) -  f_*(dT \wedge f^*S)\big)\\
& = (-1)^{\deg T} \big(d(f_*T \wedge S) - f_*(dT)\wedge S\big) = f_*T \wedge dS.
\end{aligned}
$$
So necessarily $d(f^*S) = f^*(dS)$. For commutativity with $\wedge$-products, we compute
$$\begin{aligned}
f_*(T\wedge f^*S_1 \wedge f^*S_2) & = f_*(T\wedge f^*S_1) \wedge S_2 \\
&= f_*T\wedge S_1 \wedge S_2
\end{aligned}$$
and then deduce $f^*S_1 \wedge f^*S_2 = f^*(S_1 \wedge S_2)$. We omit the verification of the remaining properties which are shown similarly.

To show existence of $f_*$, we consider the graph $Γ_f = (\mr{id},f)_*[\mbR^n, µ_{\mr{std}}]$ as a $δ$-form on $\mbR^n\times \mbR^m$. We claim that the following definition satisfies \eqref{eq:projection_formula_general}:
$$f^*S := p_{1,*}(Γ_f \wedge p_2^*S).$$
The next succession of identities verifies that claim. The first four equalities come either by definition or from the projection formula \eqref{eq:projection_main_thm}. The last equality will be explained below.
\begin{equation}
\begin{aligned}
f_*(T\wedge f^*S) & = f_*(T \wedge p_{1,*}(Γ_f \wedge p_2^*S))\\
&= f_*(p_{1,*}(p_1^*T \wedge Γ_f \wedge p_2^*S))\\
&= p_{2,*}(p_1^*T\wedge Γ_f \wedge p_2^*S)\\
&= p_{2,*}(p_1^*T\wedge Γ_f) \wedge S\\
&= f_*T \wedge S.
\end{aligned}
\end{equation}
The last equality comes from the identity $p_{2,*}(p_1^*T\wedge Γ_f) = f_*T$ which may be seen as follows. The form $p_1^*T\wedge Γ_f$ has support contained in $\Supp Γ_f$ and has the property $p_{1,*}(p_1^*T\wedge Γ_f) = T \wedge p_{1,*}Γ_f = T$ by the projection formula \eqref{eq:projection_main_thm}. As $\Supp Γ_f \to \mbR^n$ is bijective, this means $p_1^*T\wedge Γ_f = (\mr{id},f)_*(T)$. It is then merely left to note that $p_2 \circ (\mr{id},f) = f$ and the proof is complete.
\end{proof}

\begin{ex}\label{ex:general_pull_back}
Every affine linear map $f\colon \mbR^n\to \mbR^m$ can be factored as $f = h\circ g$, where $g\colon \mbR^n\to \mbR^k$ is surjective and $h\colon \mbR^k\to \mbR^m$ injective. Functoriality gives $f^* = g^*\circ h^*$ where $g^*$ is the pull-back of currents from \S\ref{ss:smooth_forms}.

Put $L = h_*[\mbR^k,μ_{\mr{std}}] \in B^{0,0,m-k}(\mbR^m)$. Then \eqref{eq:projection_formula_general} with $T = 1$ (constant function) comes out as
$$h_*h^*S = h_*(1 \wedge h^*S) = L \wedge S.$$
Especially interesting here is the property of $h^*$ to commute with $\wedge$-products, cf. Prop. \ref{propdef:pullback}. It specializes to
$$(L\wedge S_1) \wedge_L (L\wedge S_2) = L\wedge S_1 \wedge S_2,$$
where $\wedge_L$ denotes the wedge product on $L$.
\end{ex}

\subsection{Existence of the $\wedge$-product}
\label{ss:intersection_proofs}

This section proves the existence of the $\wedge$-product. We begin with some Leibnize rule properties of the product with piecewise smooth forms in \eqref{eq:def_pws_product}.

\begin{lem}\label{lem:product_pws_properties}
Let $T$ be a $δ$-form.
\begin{enumerate}[wide, labelindent=0pt, labelwidth=!, label=(\arabic*), topsep=2pt, itemsep=2pt]
\item For every homogeneous piecewise smooth form $α$ and polyhedral derivative $d_P\in \{d'_P, d''_P\}$,
$$d_P(α\wedge T) = d_Pα \wedge T + (-1)^{\deg α} α\wedge d_PT.$$
\item For every homogeneous piecewise smooth form $α \in B^{p, 0,0}$,
$$\partial' (α\wedge T) = (-1)^{\deg α} α \wedge \partial' T$$
and hence
$$d'(α\wedge T) = d'α \wedge T + (-1)^{\deg α} α \wedge d'T.$$
\item Analogously, for every homogeneous piecewise smooth form $α \in B^{0, q,0}$,
$$\partial'' (α\wedge T) = (-1)^{\deg α} α \wedge \partial'' T$$
and hence
$$d''(α\wedge T) = d''α \wedge T + (-1)^{\deg α} α \wedge d''T.$$
\end{enumerate}
\end{lem}
\begin{proof}
Identity (1) may be checked polyhedron by polyhedron and, in this way, reduces to the Leibniz rule for smooth forms. Identities (2) and (3) follow from the observation that the contractions in \eqref{eq:boundary_formula_1} are linear (up to the sign $(-1)^{\deg α}$) with respect to multiplication by piecewise smooth functions in the stated degrees.
\end{proof}

By a \emph{divisor} we mean a $d'$-closed and $d''$-closed $δ$-form of tridegree $(0,0,1)$. These are the tropical cycles with constant coefficients of codimension $1$ in classical terminology.

\begin{lem}\label{lem:divisors_are_Cartier}
Let $D$ be a divisor on $\mbR^n$. Then there exists a piecewise linear function $φ$ such that $D = d'd''φ$. It is unique up to addition of affine linear functions.
\end{lem}
\begin{proof}
This is entirely due to Lagerberg, cf. \cite{Lag}*{Prop. 5.3}, we merely give the straightforward reduction to his results. Let $U_1\subset U_2\subset \ldots$ be a covering of $\mbR^n$ by convex relatively compact opens. By definition, $D$ is a \emph{locally finite} sum of currents $m\cdot [σ,µ]$ with $m\in \mbR$. So for each $i\geq 1$, there is a finite linear combination $H_i$ of weighted hyperplanes such that $(D+H_i)\vert_{U_i}$ is \emph{positive} in the sense that all its coefficients are $\geq 0$. By \cite{Lag}*{Prop. 2.4 and Prop. 2.6}, there is then a convex function $φ'_i$ on $U_i$ such that $(D + H_i)\vert_{U_i} = d'd''φ'_i$. The lemma is easily seen to hold for hyperplanes and hence the $H_i$, so we obtain for each $i$ the existence of a convex function $φ_i$ with $D\vert_{U_i} = d'd''φ_i$. Then $φ_i$ is necessarily piecewise linear, cf. \cite{Lag}*{Proof of Prop. 5.3}. A piecewise linear function $φ$ is affine linear if and only if $d'd''φ = 0$, so the $φ_i$ are determined up to addition of affine linear functions. They may then be chosen compatibly, i.e. such that they satisfy $φ_{i+1}\vert_{U_i} = φ_i$, proving the lemma.
\end{proof}
For affine linear $φ$ and every current $T$ we have by \eqref{eq:Leibniz_current} the relation
\begin{equation}\label{eq:trivial_div_identity}
d'd''(φT) = φ\wedge d'd''T - d''φ\wedge d'T + d'φ\wedge d''T.
\end{equation}

\begin{defn}\label{def:CD_intersection}
Let $D$ be a divisor and $T$ a $δ$-form. Choose a piecewise linear function $φ$ with $D = d'd''φ$ as in Lem. \ref{lem:divisors_are_Cartier} and define
\begin{equation}\label{eq:CD_intersection}
D\cdot T := d'd''(φ\wedge T) - φ \wedge d'd''T + d''φ\wedge d'T - d'φ\wedge d''T.
\end{equation}
The definition does not depend on the choice of $φ$ by \eqref{eq:trivial_div_identity}. The resulting $D\cdot T$ is again a $δ$-form.
\end{defn}
\begin{rmk}
The definition collapses to $D\cdot T = d'd''(φT)$ whenever $d'T = d''T = 0$. This identity is well-known in Bedford--Taylor theory, cf. \cite{BT} and \cite{CLD}*{§5}.
\end{rmk}
\begin{lem}\label{lem:div_inter_simplified}
Let $D$, $T$ and $φ$ be as above. The following two identities hold:
\begin{equation}\label{eq:div_inter_simplified}
\begin{aligned}
D\cdot T & = d'(d''φ\wedge T) + d''φ\wedge d'T\\
& = -\partial'(d''φ\wedge T) - d''φ\wedge \partial'T.
\end{aligned}
\end{equation}
\end{lem}
\begin{proof}
Part (3) of Lem. \ref{lem:product_pws_properties} shows that
$$d''(φ\wedge T) = d''φ \wedge T + φ \wedge d''T.$$
Substituting this in \eqref{eq:CD_intersection} immediately leads to the first equality of \eqref{eq:div_inter_simplified}.

Part (1) of Lem. \ref{lem:product_pws_properties} together with the observation $d'_Pd''φ = 0$ implies that
$$d'_P(d''φ\wedge T) = - d''φ \wedge d'_P T.$$
Substituting this in \eqref{eq:div_inter_simplified} gives the second equality.
\end{proof}
We remark that identity \eqref{eq:div_inter_simplified} collapses to the definition of the corner locus \cite{GK}*{Def. 1.10} if $T$ is a tropical cycle.
Also, if $φ$ is affine linear, then $d''φ$ is a smooth form and sign-commutes with $\partial'$ by Lem. \ref{lem:product_pws_properties}. Then \eqref{eq:div_inter_simplified} gives $d'd''φ\cdot T = 0$ as expected. The identity also shows that if $T$ is of tridegree $(p,q,r)$, then $D\cdot T$ is of tridegree $(p,q,r+1)$. Its most important consequence for us, however, is the following simple description of $D\cdot T$.

\begin{lem}\label{lem:div_inter_combinatorial}
Let $φ$ be a piecewise linear function and $T$ a $δ$-form of tridegree $(p,q,r)$. Let further $\mcT$ be a weighted polyhedral complex subordinate to both $φ$ and $T$, say $T = \sum_{σ\in \mcT^r} α_σ\wedge [σ,µ_σ]$. Then
$$d'd''φ\cdot T = \sum_{τ\in \mcT^{r+1}} β_τ\wedge [τ,µ_τ]$$
with
\begin{equation}\label{eq:div_inter_combinatorial}
β_τ = \sum_{τ\subset σ\ \text{a facet}} \frac{\partial(φ-φ_τ)}{\partial n_{σ,τ}} α_σ\vert_τ,
\end{equation}
where $φ_τ$ is any choice of affine linear function with $(φ-φ_τ)\vert_τ$ constant.
\end{lem}
\begin{proof}
It is clear that $\mcT$ is also subordinate to $d'd''φ\cdot T$, our task is merely to determine the $β_τ$. They may be computed locally near every inner point of any given $τ$. Having some $τ$ fixed, we may replace $φ$ by $φ-φ_τ$ because $d'd''φ_τ = 0$. Then $d''(φ-φ_τ)\vert_τ = 0$, so the term $d''(φ-φ_τ) \wedge \partial' T$ in \eqref{eq:div_inter_simplified} vanishes and we are left to find the coefficient of $τ$ in $-\partial'(d''(φ-φ_τ)\wedge T)$.

Assume first that $(φ-φ_τ)\vert_σ$ is non-constant for every $σ\in \mcT^r$ containing $τ$. Then we can put $z_σ = (φ-φ_τ)\vert_σ$ to obtain \eqref{eq:div_inter_combinatorial} from a literal application of \eqref{eq:boundary_formula_3}.

The general case follows since the right hand side of \eqref{eq:div_inter_combinatorial} is a priori independent of the choice $φ_τ$ by the balancing condition \eqref{eq:def_balancing_lin_functions}.
\end{proof}

From here on, many ideas belong to Allermann--Rau \cite{AR} and we merely extend them to $δ$-forms. We will provide references to their paper for comparison.

\begin{lem}[compare \cite{AR}*{Prop. 6.7}]\label{lem:div_inter_commutative}
Given divisors $D_1$, $D_2$ and a $δ$-form $T$,
$$D_1\cdot (D_2\cdot T) = D_2\cdot (D_1\cdot T).$$
\end{lem}
\begin{proof}
Let $φ_i$ be a piecewise linear function with $D_i = d'd''φ_i$. Assume $T$ of tridegree $(p,q,r)$ and let $\mcT$ be a weighted polyhedral complex that is subordinate to $φ_1$, $φ_2$ and $T$; write $T = \sum_{σ\in \mcT^r} α_σ\wedge [σ,µ_σ]$.
Fix some $ρ\in \mcT^{r+2}$ and assume both $φ_1\vert_ρ$ and $φ_2\vert_ρ$ to vanish. Each $σ\in \mcT^r$ with $ρ\subset σ$ has precisely two facets $τ,τ'$ containing $ρ$ and we write $σ = τ + τ'$ if this relation holds. Define a constant $χ(σ)$ through $µ_σ = χ(σ)n_{τ,ρ}\wedge n_{τ',ρ}\wedge µ_ρ$ in this case. In other words, one may pick $n_{σ,τ} = χ(σ)n_{τ',ρ}$ whenever $σ = τ + τ'$. Pick an auxiliary affine linear function $z$ with $z\vert_ρ = 0$ and $z\vert_τ\neq 0$ for all $ρ\subset τ\in \mcT^{r+1}$. Define the constants
$$x_τ := \frac{\partial φ_1\vert_τ}{\partial n_{τ,ρ}},\ \ \ y_τ := \frac{\partial φ_2\vert_τ}{\partial n_{τ,ρ}},\ \ \ λ_τ := \frac{\partial z}{\partial n_{τ,ρ}},\ \ \ τ\in \mcT^{r+1},\ ρ\subset τ.$$
Then $φ_2 - (y_τ/λ_τ)z$ vanishes on $τ$ and may be used in \eqref{eq:div_inter_combinatorial} to compute the $τ$-contribution $β_τ\wedge [τ,µ_τ]$ to $D_2\cdot T$,
\begin{align}
β_τ &= \sum_{τ\subset σ\ \text{a facet}} \frac{\partial(φ_2 - (y_τ/λ_τ)z)}{\partial n_{σ,τ}} α_σ\vert_τ\\
	&= \sum_{τ\subset σ\ \text{a facet}} χ(σ)(y_{τ'} - y_τλ_{τ'}/λ_τ) α_σ\vert_τ.
\end{align}
The $ρ$-contribution $γ_ρ\wedge [ρ,µ_ρ]$ to $D_1\cdot (D_2\cdot T)$ is then, again using \eqref{eq:div_inter_combinatorial},
\begin{align}
γ_ρ &= \sum_{(τ,τ'),\ σ = τ+τ'\in \mcT^r} χ(σ)(x_τ(y_{τ'} - y_τλ_{τ'}/λ_τ))α_σ\vert_ρ\\
	&= \sum_{\{τ,τ'\},\ σ = τ+τ'\in \mcT^r} χ(σ)(x_τy_{τ'} + x_{τ'}y_τ - x_τ y_τλ_{τ'}/λ_τ - x_{τ'}y_{τ'}λ_τ/λ_{τ'})α_σ\vert_ρ.
\end{align}
The last expression is symmetric with respect to exchange of $x$ and $y$, proving the lemma.
\end{proof}

\begin{lem}\label{lem:div_inter_proj_formula}
Let $f\colon \mbR^n\to \mbR^m$ be a surjective affine linear map, $T$ a $δ$-form on $\mbR^n$ with compact support over $\mbR^m$ and $D$ a divisor on $\mbR^m$. Then the projection formula holds,
$$D\cdot f_*T = f_*(f^*D\cdot T).$$
\end{lem}
\begin{proof}
Write $D = d'd''φ$ for a piecewise linear function $φ$ as in Lem. \ref{lem:divisors_are_Cartier}. Push-forward commutes with both $d'$ and $d''$ while multiplication with the piecewise smooth forms $φ$, $d'φ$ and $d''φ$ on $\mbR^m$ in the sense of Ex. \ref{ex:pws_delta_product} obviously satisfies the projection formula. The claim now follows directly from Def. \ref{def:CD_intersection}.
\end{proof}

\begin{lem}[compare \cite{AR}*{Lem. 9.4}]\label{lem:div_inter_diagonal}
Let $T$ be a $δ$-form on $\mbR^c\times \mbR^m$ and denote by $D_i = p_{12}^*(d'd''\max\{x_i,y_i\})$ the divisor on $\mbR^c\times \mbR^c\times \mbR^m$ where the $i$-th coordinates of the first two factors agree, $i = 1,\ldots,c$. Let $g(x,z) := (x,x,z)$ be the partial diagonal $\mbR^c\times \mbR^m \to  \mbR^c\times \mbR^c\times \mbR^m$. Then
$$D_1\cdots D_c \cdot (\mbR^c\times T) = g_*T.$$
\end{lem}
\begin{proof}
By a recursive argument, it is enough to treat the case $c = 1$. Let $\mcT$ be a polyhedral complex on $\mbR \times \mbR^m$ that is subordinate to $T$, say $T = \sum_{σ\in \mcT} α_σ\wedge [σ,µ_σ]$. Assume without loss of generality that $\bigcup_{σ\in \mcT} σ = \mbR\times \mbR^m$ and define, for each $σ$,
$$\wt{σ}_{?} = \{(x,y,z)\in \mbR\times σ\mid x\ ?\ y\},\ \ \ ? \in \{\geq, \leq\}.$$
Let $x$ and $y$ denote the coordinates on the first two factors of $\mbR\times \mbR\times \mbR^m$. A polyhedral complex structure on $\mbR\times \mbR\times \mbR^m$ that is subordinate to both $\mbR\times T$ and $φ = \max\{0,x-y\}$ is then, for example,
$$\mcS = \bigcup_{σ\in \mcT} \{\wt{σ}_{\geq}, g(σ), \wt{σ}_\leq\}.$$
It becomes a weighted complex by endowing $\wt{σ}_{\geq}$ and $\wt{σ}_{\leq}$ with weight $µ_{\mr{std}}\wedge µ_σ$ and $g(σ)$ with $g(µ_σ)$. 
The support of $d'd''φ\cdot (\mbR\times T)$ is contained in $φ$'s locus of non-linearity $\{x = y\} = \bigcup_{σ\in \mcT} g(σ)$. Given a polyhedron $g(σ)\in \mcS$, it is the facet of precisely the polyhedra $\wt{σ}_{\geq}$, $\wt{σ}_{\leq}$ and all $g(ρ)$ such that $σ\subset ρ$ is a facet. Normal vectors in these cases are $(1,0,0)$, $(-1,0,0)$ and $g(n_{ρ,σ})$, respectively. Using that $φ\vert_{\{x=y\}} = 0$, the contribution $β_σ \wedge [g(σ),g(µ_σ)]$ of $g(σ)$ to $d'd''φ\cdot (\mbR\times T)$ is by Lem. \ref{lem:div_inter_combinatorial}
$$β_σ = \left(\frac{\partial φ\vert_{x \geq y}}{\partial (1,0,0)} + \frac{\partial φ\vert_{x \leq y}}{\partial (-1,0,0)}\right) g_*α_σ + \sum_{σ\subset ρ\ \text{a facet}} \frac{\partial φ\vert_{g(ρ)}}{\partial g(n_{ρ,σ})} g_*(α_ρ\vert_σ).$$
Since $φ\vert_{\{x\leq y\}} = 0$, only the first term is non-zero and contributes $g_*α_σ$ as claimed.
\end{proof} 

\begin{defn}[compare \cite{AR}*{Def. 9.3}]\label{def:delta_form_inter_prod}
Let $D_i = d'd''(\max\{x_i,y_i\})$ denote the divisor on $\mbR^n\times \mbR^n$ where the $i$-th coordinates agree.
The \emph{$\wedge$-product of $δ$-forms $S,T\in B(\mbR^n)$} is defined as the $δ$-form
$$S\wedge T := p_{1,*}(D_1 \cdots D_n\cdot (S\times T)).$$
The notational convention (and only possibility) here is that the successive product is evaluated from right to left. We also write $Δ\cdot (S\times T)$ instead of $D_1\cdots D_n \cdot (S\times T)$. Note that we have already seen that the order of the $D_i$ does not matter, but only Cor. \ref{cor:div_inter_properties} below will prove the independence of the choice of $\{D_1,\ldots,D_n\}$ to describe the diagonal.
\end{defn}

\begin{lem}\label{lem:tensor_of_delta_forms_is_intersection}
The following identity holds for all $δ$-forms $S$ and $T$,
$$S\times T = p_1^*S\wedge p_2^*T.$$
\end{lem}
\begin{proof}
Let $g\colon \mbR^n\times \mbR^n \to (\mbR^n\times \mbR^n)\times (\mbR^n\times \mbR^n)$ be the diagonal.
The first equality in the following is by definition, the second is Lem. \ref{lem:div_inter_diagonal} and the third is the identity $p_{12}\circ g = \mr{id}$.
$$\begin{aligned}
p_1^*S\wedge p_2^*T &= p_{12,*} (Δ\cdot (S\times \mbR^n\times \mbR^n\times T))\\
&= p_{12,*} g_*(S\times T)\\
&= S\times T.
\end{aligned}
$$
\end{proof}

\begin{lem}[compare \cite{AR}*{Lem. 9.7}]\label{lem:div_inter_associat}
Let $S, T$ be $δ$-forms on $\mbR^n$ and $C$ a divisor. Then
$$C\cdot (S\wedge T) = (C\cdot S)\wedge T.$$
\end{lem}
\begin{proof}
It follows from Lem. \ref{lem:div_inter_combinatorial} that $p_1^*C \cdot (S\times T) = (C \cdot S)\times T$. Then the claim follows from the commutativity in Lem. \ref{lem:div_inter_commutative} and the projection formula in Lem. \ref{lem:div_inter_proj_formula}:
$$
\begin{aligned}
C\cdot p_{1,*}(D_1\cdots D_n\cdot(S\times T)) &= p_{1,*}(p_1^*C\cdot D_1\cdots D_n\cdot (S\times T))\\
 &= p_{1,*}(D_1\cdots D_n\cdot ((C\cdot S)\times T)).
\end{aligned}
$$
\end{proof}
Successive application of Lem. \ref{lem:div_inter_associat} shows the following corollary.
\begin{cor}[compare \cite{AR}*{Cor. 9.8}]\label{cor:div_inter_properties}
Let $T$ be a $δ$-form on $\mbR^n$ and $C,C_1,\ldots,C_l$ divisors. Then
$$C_1\cdots C_l\cdot T = (C_1\cdots C_l)\wedge T.$$
In particular,
$$C \wedge T = C\cdot T,\ \ \ Δ\wedge (S\times T) = Δ \cdot (S\times T).$$
\end{cor}

\begin{proof}[Proof of Thm. \ref{thm:main}.] So far, Def. \ref{def:delta_form_inter_prod} provides a well-defined bilinear map $\wedge\colon  B\times B \to  B$. It is left to verify all the properties stated in Thm. \ref{thm:main}. 

(a) The $\wedge$-product is clearly trihomogeneous in the sense that it restricts to maps $B^{p,q,r}\times B^{s,t,u}\to B^{p+s,q+t,r+u}$. It is graded-commutative in the sense $S\wedge T = (-1)^{\deg S \deg T} T\wedge S$ for homogeneous $S$ and $T$ because
$$s^*(S\times T) = (-1)^{\deg S\deg T} T\times S,$$
where $s\colon \mbR^n\times\mbR^n\to \mbR^n\times \mbR^n$ is the map that switches the two factors.

(b) We claim that the $\wedge$-product satisfies the projection formula
$$S\wedge f_*T = f_*(f^*S\wedge T)$$
for every \emph{surjective} linear map $f\colon \mbR^n\to \mbR^m$. To check this, we may assume $n = m + c$ and, after a change of coordinates, $f\colon \mbR^m\times \mbR^c \to  \mbR^m$ being just the projection. Then simply $f^*S = S\times \mbR^c$. Recall that $D_i = d'd''\max\{x_i,y_i\}$ on $\mbR^n\times \mbR^n$ and $Δ = D_1\cdots D_n$. Write $Δ^m = D_1\cdots D_m$ and $Δ^c = D_{m+1}\cdots D_n$. Then the following equalities hold, as will be explained below.
\begin{equation}
\begin{aligned}
f_*(f^*S\wedge T) &= p_{1,*}(f,f)_*(Δ^m\cdot Δ^c \cdot (f^*S\times T))\\
&= p_{1,*} \big(δ^m \cdot (f,f)_*(Δ^c\cdot (S\times \mbR^c\times T))\big)\\
&= p_{1,*} \big(δ^m \cdot (f,f)_*(S \times (Δ^c \cdot (\mbR^c\times T)))\big)\\
&= p_{1,*} \big(δ^m \cdot (f,f)_*(S \times g_*T)\big)\\
&= S\wedge f_*T.
\end{aligned}
\end{equation}
The first equality is the definition of the left hand side combined with the identity $f\circ p_1 = p_1 \circ (f,f)$. The second follows from the projection formula for divisor intersection, Lem. \ref{lem:div_inter_proj_formula}, applied to $Δ^m = (f,f)^*δ^m$, where $δ^m\subset \mbR^m\times \mbR^m$ denotes the diagonal. The third equality is the observation $p_2^*C \cdot (X \times Y) = X \times (p_2^* C \cdot Y)$ for any divisor $C$ and $δ$-forms $X,Y$, applied successively to the divisor intersection $Δ^c$. The map $g$ in the next line is the partial diagonal
$$g\colon \mbR^m\times \mbR^c \longrightarrow  \mbR^c \times \mbR^m\times \mbR^c,\ (x,y)\longmapsto (y, x, y)$$
and the identification $Δ^c \cdot (\mbR^c\times T) = g_*T$ is Lem. \ref{lem:div_inter_diagonal}. The final equality then is the observation $(f,f)_*(S \times g_*T) = S\times f_*T$ which follows from the identity $(f,f)\circ (\mr{id}_{\mbR^m},g) = (\mr{id}_{\mbR^m}, f)$.

(c) The next claim is that the $\wedge$-product is associative, $S\wedge (T\wedge U) = (S\wedge T)\wedge U$. Indeed, applying the projection formula (b) repeatedly,	one obtains
$$S\wedge (T\wedge U) = p_{1,*}\big(Δ \cdot (S\times \mbR^n\times T\times U)\big)$$
where the intersection takes place on $(\mbR^n)^4$ and where $Δ = \prod_{i = 1}^n p_{12}^*D_i \cdot p_{23}^*D_i\cdot p_{34}^*D_i$ is the diagonal $\mbR^n\subset (\mbR^n)^4$. In exactly the same way,
$$(S\wedge T)\wedge U = p_{1,*}\big(Δ\cdot (S\times T\times \mbR^n\times U)\big).$$
The two expressions are seen to be equal by switching the middle factors as in Step (a).

(d) Next, we claim that $α\wedge T = αT$ for every piecewise smooth $α$. This follows from Lem. \ref{lem:div_inter_diagonal} and the fact that multiplication by piecewise smooth forms commutes with divisor intersection. The latter is immediate from Lem. \ref{lem:div_inter_combinatorial}.

(e) Lem. \ref{lem:tensor_of_delta_forms_is_intersection} furthermore showed that $S\times T = p_1^*S \wedge p_2^*T$, so the constructed $\wedge$-product is computed by intersection with the diagonal, cf. \eqref{eq:inter_diagonal}.

(f) We turn to the Leibniz rule. Let $C$ be a divisor and $T$ a $δ$-form. Our first step is to prove the identity
\begin{equation}\label{eq:Leibniz_divisor_inter}
d(C\wedge T) = C \wedge dT,\ \ \ d\in \{d',d''\}.
\end{equation}
We write $d = d_P - \partial$, with $d_P \in \{d_P', d_P''\}$ and $\partial \in \{\partial',\partial''\}$ suitable, and verify \eqref{eq:Leibniz_divisor_inter} for $d_P$ and $\partial$ separately.

Applying \eqref{eq:div_inter_combinatorial}, the identity $d_P(C\wedge T) = C \wedge d_PT$ is immediate. Writing $D = d'd''φ$ for a piecewise linear function $φ$ as in Lem. \ref{lem:divisors_are_Cartier} and using \eqref{eq:div_inter_simplified} twice, we have
$$\partial'(C\wedge T) = - \partial'\partial'(d''φ\wedge T) - \partial'(d''φ\wedge \partial'T) = C \wedge \partial'T$$
because $(\partial')^2 = 0$ by Lem. \ref{lem:differential_calculus_delta_forms}.
Finally, applying Lem. \ref{lem:product_pws_properties} (3) and the rule $\partial'\partial'' = - \partial''\partial'$, we also obtain
$$\begin{aligned}
\partial''(C\wedge T) & = - \partial''\partial'(d''φ \wedge T) - \partial''(d''φ\wedge \partial'T) \\
&= - \partial'(d''φ \wedge \partial''T) - (d''φ\wedge \partial' \partial'' T) = C\wedge \partial''T.
\end{aligned}$$
This finishes the proof of \eqref{eq:Leibniz_divisor_inter}. Successive application of the divisor case now yields
$$d(\Delta \wedge (S\times T)) = \Delta \wedge d(S\times T).$$
The Leibniz rule \eqref{eq:tensor_product_current_derivative_d} for exterior products, coupled with Def. \ref{def:delta_form_inter_prod}, completes the proof of the Leibniz rule for $d$ in general. Separating by tridegree provides the Leibniz rules for the other differential operators.

(g) The identity $f^*(S\wedge T) = f^*S\wedge f^*T$ only uses the fact $p_1^*D\cdot (S \times T) = (D\cdot S)\times T$ for divisor intersection. Namely assume $f\colon \mbR^m\times \mbR^c \to \mbR^m$ to be the projection and write $p_{12}\colon \mbR^m\times \mbR^c\times \mbR^m\times \mbR^c \to  \mbR^m\times \mbR^c$ for the projection to the first two factors. Then, in the terminology of Step (b),
\begin{equation}
\begin{aligned}
f^*S\wedge f^*T &= p_{12,*}(Δ^m\cdot Δ^c \cdot (S \times \mbR^c\times T \times \mbR^c))\\
&= p_{12,*} (Δ^m \cdot (S\times T)) \times Δ^c\\
&= (S\wedge T) \times \mbR^c = f^*(S\wedge T).
\end{aligned}
\end{equation}

(h) Finally, the tropical intersection products of Allermann--Rau \cite{AR}, its extension to smoothly weighted rational polyhedra in \cite{GK}*{Rmk. 1.4}, and the intersection product of Esterov \cite{Esterov} for polynomially weighted (possibly non-rational) polyhedra can all be expressed in terms of divisor intersection and restriction to the diagonal. In these two specific cases, they coincide with our definition. So any two of the mentioned products coincide whenever both are defined.
\end{proof}

\subsection{Fan Displacement Rule}
\label{ss:fan_displacement}

Two linear subspaces $N_1,N_2\subseteq \mbR^n$ are said to intersect \emph{transversally} if $N_1 + N_2 = \mbR^n$. (Equivalently, their intersection is transversal if $\codim(N_1\cap N_2) = \codim(N_1) + \codim(N_2)$.) In the transversal case, there is an exact sequence
$$0 \to N_1\cap N_2 \to N_1 \oplus N_2 \to \mbR^n \to 0.$$
Given weights $µ_1$ and $µ_2$ for $N_1$ and $N_2$, respectively, we denote by $µ_1\cap µ_2$ the weight on $N_1\cap N_2$ that satisfies $(µ_1\cap µ_2)\wedge µ_{\mr{std}} = µ_1\wedge µ_2$ in the sense of \eqref{eq:prod_of_weight}. The next lemma is easily checked.

\begin{lem}\label{lem:inter_subspace}
Let $[N_1,µ_1],[N_2,µ_2]\subseteq \mbR^n$ be weighted linear subspaces, viewed as $δ$-forms. Assume that their intersection is transverse. Then
$$[N_1,µ_1]\wedge [N_2,µ_2] = [N_1\cap N_2, µ_1\cap µ_2].$$
\end{lem}

Let $\mcT_1$ and $\mcT_2$ be polyhedral complexes on $\mbR^n$ which are \emph{pure of codimension $r_1$ and $r_2$}, respectively. By this we mean that $\mcT_i$ agrees with the set of faces of all $σ_i\in \mcT_i^{r_i}$. Then $\mcT_1$ and $\mcT_2$ are said to \emph{intersect transversally} if, for all pairs $(σ_1,σ_2)\in \mcT^{r_1}\times \mcT^{r_2}$, the intersection $σ_1\cap σ_2$ is either empty or of codimension $r_1+r_2$ and not contained in the union of boundaries $\partial σ_1\cup \partial σ_2$. Note that then $N_{σ_1}$ and $N_{σ_2}$ intersect transversally whenever $σ_1\cap σ_2 \neq \emptyset$.

Assume the above $\mcT_i$ to intersect transversally and let $\mcS$ be the polyhedral complex of all $σ_1\cap σ_2$, $σ_i \in \mcT_i$. Then $\mcS$ is pure of codimension $r_1+r_2$ and every top-dimensional $τ\in \mcS^{r_1+r_2}$ determines a \emph{unique} pair $(σ_1,σ_2)\in \mcT^{r_1}_1\times \mcT^{r_2}_2$ such that $τ = σ_1\cap σ_2$.

\begin{lem}\label{lem:inter_transversal}
Let $\mcT_1$ and $\mcT_2$ be transversally intersecting weighted polyhedral complexes of pure codimensions $r_1$ and $r_2$, respectively. Let
$$T_1 = \sum_{σ \in \mcT_1^{r_1}} α_σ \wedge [σ,µ_σ]\ \ \ \text{and}\ \ \ T_2 = \sum_{σ\in \mcT_2^{r_2}} β_σ \wedge [σ,ν_σ]$$
be $δ$-forms. Then
\begin{equation}\label{eq:inter_transversal}
T_1\wedge T_2 = \sum_{(σ_1,σ_2)\in \mcT^{r_1}\times \mcT^{r_2},\ σ_1\cap σ_2 \neq \emptyset} α_{σ_1}\wedge β_{σ_2} \wedge [σ_1\cap σ_2, µ_{σ_1}\cap ν_{σ_2}].
\end{equation}
\end{lem}
\begin{proof}
Let $\mcS$ be the polyhedral complex generated by all $σ_1\cap σ_2,\ σ_i\in \mcT_i^{r_i}$. Then
$$\Supp (T_1\wedge T_2) \subseteq \Supp T_1\cap \Supp T_2 \subseteq \bigcup_{τ\in \mcS^{r_1+r_2}} τ,$$
so $T_1\wedge T_2 = \sum_{τ\in \mcS^{r_1+r_2}} γ_τ\wedge [τ,µ_τ]$ for certain forms $γ_τ$ and weights $µ_τ$. Each $γ_τ$ is uniquely determined by its restriction to the relative interior $τ^\circ\subseteq τ$. Also, every occurring $τ$ is in a unique way the intersection $σ_1\cap σ_2$ of top-dimensional $σ_i\in \mcT_i$. On an open neighborhood of $τ^\circ$, the situation then agrees with a subspace intersection as in Lem. \ref{lem:inter_subspace}, multiplied by $α_{σ_1}\wedge β_{σ_2}$, and the claim follows from Thm. \ref{thm:main} and Lem. \ref{lem:inter_subspace}.
\end{proof}

Recall that a sequence (resp. net) of currents $(T_i)_{i\in I}$ \emph{converges weakly} to a current $T$ if for every test form $η\in A_c$, the sequence (resp. net) $T_i(η)$ converges to $T(η)$. Given a current $T$ on $\mbR^n$ and a vector $v\in \mbR^n$, we write $λ_v(x) = x + v$ and denote by $v+T = λ_{v,*}T=λ_{-v}^*T$ the $v$-translated current.

\begin{prop}\label{prop:limit_product}
Let $S,T\in B(\mbR^n)$ be $δ$-forms and $v\in \mbR^n$ a vector. Then there is the weak convergence
$$S\wedge (εv+T) \longrightarrow S\wedge T,\ \ \ ε\longrightarrow 0.$$
\end{prop}
\begin{proof}
Consider on $\mbR$ the piecewise smooth form $ρ_ε = - d''x \wedge \big[[0,ε], µ_{\mr{std}}\big]$. Its boundary $\partial'ρ_ε = δ_0 - δ_ε$ is the difference of the Dirac measures at $0$ and $ε$. Next, consider the map $f\colon \mbR\times \mbR^n \to \mbR^n,\ (ε,y)\mapsto  y - εv$. The Leibniz rule yields
\begin{equation}\label{eq:omega}
ω_ε := p_1^*(δ_0 - δ_ε)\wedge f^*T = \partial' (p_1^*ρ_ε \wedge f^*T) + p_1^*ρ_ε\wedge f^*\partial'T.
\end{equation}
Note that $p_1^*ρ_ε$ is piecewise smooth, making the $\wedge$-products on the right hand side straightforward, cf. Ex. \ref{ex:pws_delta_product}. Now Lem. \ref{lem:inter_transversal} implies that
$$p_1^*δ_{ε}\wedge f^*T = \{ε\} \times (εv + T),\ \ \ ε\in \mbR,$$
and hence
\begin{equation}
\label{eq:weak_convergence_difference}
S\wedge (T - (εv + T)) = S\wedge p_{2,*} ω_ε.
\end{equation}
Our task is to show that this expression converges weakly to $0$ as $ε\to 0$. The projection formula, cf. Thm. \ref{thm:main}, allows to rewrite \eqref{eq:weak_convergence_difference} as
$$p_{2,*}(p_2^*S\wedge ω_ε).$$
Now observe that if $(X_i)_{i\in I}\to X$ is a weakly convergent net of currents on $\mbR\times \mbR^n$ with $X$ and all $X_i$ of compact support over $\mbR^n$, then $(p_{2,*}X_i)_{i\in I} \to p_{2,*}X$ by definition \eqref{eq:push_for_currents}. So it remains to show $p_2^*S\wedge ω_ε\to 0$ as $ε\to 0$.

\emph{Claim.} For \emph{every} polyhedral current $γ\in P(\mbR\times \mbR^n)$, there is the weak convergence $p_1^*ρ_ε\wedge γ \to 0$ as $ε\to 0$. This is straightforward: It is enough to consider the case $γ = α\wedge [σ,µ]$ in which case there are the two possibilities that $p_1(σ)$ is of dimension $0$ or $1$. In the $0$-dimensional case, $p_1^*ρ_ε\vert_{σ} = 0$ and we are done. In the $1$-dimensional case, a simple volume argument shows that for every compactly supported $(\dim σ,\dim σ)$-form $η$,
$$\int_{\big[σ\, \cap\, ([0,\,ε]\times \mbR^n),\, µ\big]}η \longrightarrow 0,\ \ \ ε\longrightarrow 0,$$
which implies the claim.

It follows that $p_2^*S \wedge p_1^*ρ_ε \wedge f^*\partial'T \to 0$ as $ε\to 0$ and it is only left to show, cf. \eqref{eq:omega}, that
\begin{equation}\label{eq:convergence_to_show}
p_2^*S \wedge \partial' (p_1^*ρ_ε \wedge f^*T) \longrightarrow 0,\ \ \ ε\longrightarrow 0.
\end{equation}
Applying the Leibniz rule for $\partial'$ and using again the above claim, one is reduced to proving
$$\partial' \big(p_2^*S \wedge p_1^*ρ_ε \wedge f^*T\big) \longrightarrow 0,\ \ \ ε\longrightarrow 0.$$
Now for every weakly convergent net of currents $(X_i)_{i\in I} \to X$, the sequence of derivatives $(d'X_i)_{i\in I} \to d'X$ converges weakly, which follows immediately from the definition in \eqref{eq:signs_for_currents}. Using the above claim once more, it is hence enough to show
$$d'_P \big(p_2^*S \wedge p_1^*ρ_ε \wedge f^*T\big) \longrightarrow 0,\ \ \ ε\longrightarrow 0.$$
Applying the Leibniz rule for $d'_P$ and the vanishing $d'_Pρ_ε = 0$, this last statement follows from yet another application of the above claim. The proof is complete.
\end{proof}

Let $\mcT_1$ and $\mcT_2$ be \emph{finite} polyhedral complexes on $\mbR^n$, pure of codimensions $r_1$ and $r_2$ respectively. A vector $v\in \mbR^n$ is called \emph{generic} for $\mcT_1$ and $\mcT_2$ if there exists $ε_0>0$ such that $\mcT_1$ and $εv + \mcT_2$ intersect transversally for all $0<ε<ε_0$. For not necessarily finite $\mcT_1$ and $\mcT_2$, a vector $v$ is called generic if it is generic for all finite subcomplexes of the same pure codimensions. Generic vectors in this sense always exist.

\begin{construction}\label{constr:fan_displacement}
Let $v$ be a generic vector for two polyhedral complexes $\mcT_1$ and $\mcT_2$ that are pure of codimensions $r_1$ and $r_2$, respectively. Let
$$T_1 = \sum_{σ\in \mcT_1^{r_1}} α_σ\wedge [σ,µ_σ],\ \ \ T_2 = \sum_{σ\in \mcT_2^{r_2}} β_σ \wedge [σ,ν_σ]$$
be $δ$-forms. Define their \emph{$v$-displacement product} with respect to $\mcT_1$ and $\mcT_2$ as
\begin{equation}\label{eq:v_displacement_product}
T_1\cdot_v T_2 := \sum_{(σ_1,σ_2)\in \mcT_1^{r_1}\times \mcT_2^{r_2},\ σ_1\cap (εv+ σ_2)\neq \emptyset\ \text{for }ε\searrow 0} α_{σ_1}\wedge β_{σ_2} \wedge [σ_1\cap σ_2, µ_{σ_1}\cap ν_{σ_2}].
\end{equation}
The sum here is over all $(σ_1,σ_2)$ such that $σ_1\cap (εv + σ_2)\neq \emptyset$ for all sufficiently small $ε$. Note that $v$ need not be generic for subdivisions of $\mcT_1$ and $\mcT_2$ anymore, which is why the definition depends on their choice.
\end{construction}

\begin{prop}\label{prop:fan_displacement}
Let $T_1$ and $T_2$ be $δ$-forms with subordinate polyhedral complexes $\mcT_1$ and $\mcT_2$ as above. Assume $v$ is generic for $\mcT_1$ and $\mcT_2$. Then the $v$-displacement product (with respect to the $\mcT_i$) computes the $\wedge$-product,
\begin{equation}\label{eq:fan_displacement}
T_1\cdot_v T_2 = T_1\wedge T_2.
\end{equation}
In particular, the $v$-displacement product is independent of the choices $\mcT_1$, $\mcT_2$ and $v$.
\end{prop}
\begin{proof}
Both sides of \eqref{eq:fan_displacement} are computed locally, so we may assume $\mcT_1$ and $\mcT_2$ to be finite by a partition of unity argument. Lem. \ref{lem:inter_transversal} then shows that for all sufficiently small $ε>0$,
$$T_1\wedge (εv + T_2) = \sum_{(σ_1,σ_2)\in \mcT_1^{r_1}\times \mcT_2^{r_2},\ σ_1\cap (εv+ σ_2)\neq \emptyset} α_{σ_1}\wedge (εv + β_{σ_2}) \wedge [σ_1\cap (εv + σ_2), µ_{σ_1}\cap ν_{σ_2}].$$
The intersection $σ_1\cap (εv + σ_2)$ being non-empty and transverse for all sufficiently small $ε$ implies that $σ_1\cap σ_2$ is non-empty and of codimension $r_1 + r_2$. Moreover in this case,
$$α_{σ_1}\wedge (εv + β_{σ_2}) \wedge [σ_1\cap (εv + σ_2), µ_{σ_1}\cap ν_{σ_2}] \longrightarrow α_{σ_1}\wedge β_{σ_2} \wedge [σ_1\cap σ_2, µ_{σ_1}\cap ν_{σ_2}]$$
in the weak sense. It follows that $T_1\wedge (εv+T_2) \to T_1\cdot_v T_2$ in the weak sense.
Prop. \ref{prop:limit_product} on the other hand shows that this limit equals $T_1\wedge T_2$, proving the proposition.
\end{proof}

\end{document}